\documentclass[11pt]{article}

\usepackage[utf8]{inputenc}
\usepackage[T1]{fontenc}

 \usepackage[a4paper,
   margin=2.5cm,
   top=2cm,
   bottom=2cm,
   includefoot,
   footskip=20pt,
 ]{geometry}

\usepackage{layout}

\usepackage{hyperref}
\hypersetup{hypertexnames = false, bookmarksdepth = 2, bookmarksopen = true, colorlinks, linkcolor = black, citecolor = black, urlcolor = blue, pdfstartview={XYZ null null 1}}

\usepackage{amsfonts}
\usepackage{amsmath}
\usepackage{amsthm}
\usepackage{eucal}
\usepackage{amssymb}
\usepackage[dvipsnames]{xcolor}
\usepackage{bm}
\usepackage{graphicx}
\usepackage[ps,all,arc,rotate,cmtip]{xy}

\usepackage{chngcntr}

\usepackage[capitalise]{cleveref}

\usepackage{colonequals}
\usepackage{mathtools}
\usepackage{thmtools}
\usepackage{tikz-cd, quiver}
\usepackage[colorinlistoftodos, color=green!40!white, textsize = footnotesize]{todonotes}
\usepackage{xparse}
\usepackage{xspace}
\usepackage{color}

\usepackage{enumitem}

\usepackage[normalem]{ulem}

\usepackage{mleftright}

\usepackage{microtype}
\usepackage{lmodern}
\theoremstyle{plain}
\newtheorem{theorem}[subsubsection]{Theorem}
\newtheorem{corollary}[theorem]{Corollary}
\newtheorem{lemma}[theorem]{Lemma}
\newtheorem{proposition}[theorem]{Proposition}
\newtheorem{mainthm}{Theorem}

\theoremstyle{definition}
\newtheorem{definition}[theorem]{Definition}

\newtheorem{remark}[theorem]{Remark}

\newcommand{\Ac}{\mathcal{A}}

\newcommand{\Cc}{\mathcal{C}}
\newcommand{\Ec}{\mathcal{E}}

\newcommand{\Gc}{\mathcal{G}}
\newcommand{\Lc}{\mathcal{L}}
\newcommand{\Mc}{\mathcal{M}}

\newcommand{\Oc}{\mathcal{O}}
\newcommand{\Uc}{\mathcal{U}}

\newcommand{\Xc}{\mathcal{X}}
\newcommand{\Yc}{\mathcal{Y}}

\newcommand{\Bun}{\mathcal{B}\mathrm{un}}
\renewcommand{\AA}{\mathbb{A}}
\newcommand{\PP}{\mathbb{P}}

\newcommand{\RR}{\mathbb{R}}
\newcommand{\ZZ}{\mathbb{Z}}
\newcommand{\NN}{\mathbb{N}}

\newcommand{\m}{\underline{m}}
\newcommand{\stackypoints}{\underline{\textit{p}}}

\newcommand{\Knum}{\mathrm{K}^{\mathrm{num}}_0}

\newcommand{\Knot}{\mathrm{K}_0}

\newcommand{\Lb}{\mathbf{L}}
\newcommand{\Mr}{M}
\newcommand{\Rb}{\mathbf{R}}
\newcommand{\Rsf}{\mathsf{R}}

\newcommand{\alphab}{{\bm{\alpha}}}
\newcommand{\betab}{{\bm{\beta}}}
\newcommand{\gammab}{{\bm{\gamma}}}
\newcommand{\deltab}{{\bm{\delta}}}
\newcommand{\etab}{{\bm{\eta}}}

\newcommand{\omegab}{{\bm{\omega}}}

\newcommand{\alphaw}{{\widetilde {\bm{\alpha}} }}
\newcommand{\betaw}{{\widetilde {\bm{\beta}} }}

\newcommand{\deltaw}{{\widetilde {\bm{\delta}} }}
\newcommand{\etaw}{{\widetilde {\bm{\eta}} }}

\DeclareMathOperator{\car}{char}
\DeclareMathOperator{\coker}{coker}
\DeclareMathOperator{\Coh}{\mathcal{C}oh}

\DeclareMathOperator{\Dqc}{D_{qc}}
\DeclareMathOperator\ev{ev}

\DeclareMathOperator\Ext{Ext}
\DeclareMathOperator\Extstack{\mathcal{E}xt}
\DeclareMathOperator\Hom{Hom}
\DeclareMathOperator\im{Im}

\DeclareMathOperator\localHom{\mathcal{H}om}
\DeclareMathOperator\localRHom{\Rb\!\mathcal{H}om}
\DeclareMathOperator\rk{rank}
\newcommand{\SD}{\operatorname{SD}}
\DeclareMathOperator\RHom{\Rb\!{H}om}
\DeclareMathOperator{\Spec}{Spec}

\newcommand{\Aut}{\operatorname{Aut}}

\newcommand{\rank}{\operatorname{rank}}
\newcommand{\ext}[4][]{\operatorname{Ext}_{#1}^{#2}\mleft(#3, #4\mright)}
\newcommand{\isom}[2]{\underline{\operatorname{Isom}}\mleft(#1, #2\mright)}
\newcommand{\littleext}[4][]{\operatorname{ext}_{#1}^{#2}\mleft(#3, #4\mright)}

\newcommand{\pic}[1]{\operatorname{Pic}_{#1}}	
\newcommand{\Pic}{\mathcal{P}\mathrm{ic}} 

\mathchardef\mhyphen="2D
\newcommand{\sst}{\mhyphen\mathrm{ss}}
\newcommand{\st}{\mhyphen\mathrm{s}}

\newcommand{\blank}{\, \underline{\;\;}\, }
\newcommand{\Gm}{\mathbb{G}_\mathrm{m}}
\newcommand{\ST}{\overline{\mathrm{ST}}}

\makeatletter
\def\IfEmptyTF#1{%
  \if\relax\detokenize{#1}\relax
    \expandafter\@firstoftwo
  \else
    \expandafter\@secondoftwo
  \fi}
\makeatother

\DeclareDocumentCommand\MR{o}
{ \IfNoValueTF{#1}{\Rsf_{\betab}}{\Rsf_{#1}}}

\DeclareDocumentCommand\Mstack{o}
{ \IfNoValueTF{#1}{\Bun_{\betab}}{\Bun_{#1}}}

\DeclareDocumentCommand\Mstacka{o}
{ \IfNoValueTF{#1}{\Bun^{\alphab\sst}}{\Bun_{#1\sst}}}

\DeclareDocumentCommand\Mspace{o}
{ \IfNoValueTF{#1}{\Mr_{\betab}}{\Mr_{#1}}}

\DeclareDocumentCommand\Mstackss{oo}
{\IfNoValueTF{#1}
	{\IfNoValueTF{#2}
        {\Mc^{\alphab\sst}_{\betab}}
        {\IfEmptyTF{#2}
            {\Mc^{\alphab\sst}_{\betab}}
            {\Mc^{#2\sst}_{\betab}}
        }
    }
    {\IfEmptyTF{#1}
	   {{\IfNoValueTF{#2}
            {\Mc^{\alphab\sst}_{\betab}}
            {\IfEmptyTF{#2}
                {\Mc^{\alphab\sst}_{\betab}}
                {\Mc^{#2\sst}_{\betab}}
            }
        }
    }{\IfEmptyTF{#2}
        {{\Mc^{\alphab\sst}_{#1}}}
        {\Mc^{#2\sst}_{#1}}}}
}

\DeclareDocumentCommand\Mstackst{oo}
{\IfNoValueTF{#1}
	{\IfNoValueTF{#2}
        {\Mc^{\alphab\st}_{\betab}}
        {\IfEmptyTF{#2}
            {\Mc^{\alphab\st}_{\betab}}
            {\Mc^{#2\sst}_{\betab}}
        }
    }
    {\IfEmptyTF{#1}
	   {{\IfNoValueTF{#2}
            {\Mc^{\alphab\st}_{\betab}}
            {\IfEmptyTF{#2}
                {\Mc^{\alphab\st}_{\betab}}
                {\Mc^{#2\sst}_{\betab}}
            }
        }
    }{\IfEmptyTF{#2}
        {{\Mc^{\alphab\st}_{#1}}}
        {\Mc^{#2\st}_{#1}}}}
}

\DeclareDocumentCommand\Mspacess{oo}
{\IfNoValueTF{#1}
	{\IfNoValueTF{#2}
        {\Mr^{\alphab\sst}_{\betab}}
        {\IfEmptyTF{#2}
            {\Mr^{\alphab\sst}_{\betab}}
            {\Mr^{#2\sst}_{\betab}}
        }
    }
    {\IfEmptyTF{#1}
	   {{\IfNoValueTF{#2}
            {\Mr^{\alphab\sst}_{\betab}}
            {\IfEmptyTF{#2}
                {\Mr^{\alphab\sst}_{\betab}}
                {\Mr^{#2\sst}_{\betab}}
            }
        }
    }{\IfEmptyTF{#2}
        {{\Mr^{\alphab \sst}_{#1}}}
        {\Mr^{#2\sst}_{#1}}}}
}

\DeclareDocumentCommand\MRst{oo}
{\IfNoValueTF{#1}
	{\IfNoValueTF{#2}
        {\Rsf^{\alpha\st}_{\beta}}
        {\IfEmptyTF{#2}
            {\Rsf^{\alpha\st}_{\beta}}
            {\Rsf^{#2\sst}_{\beta}}
        }
    }
    {\IfEmptyTF{#1}
	   {{\IfNoValueTF{#2}
            {\Bun^{\alpha\st}_{\beta}}
            {\IfEmptyTF{#2}
                {\Rsf^{\alpha\st}_{\beta}}
                {\Rsf^{#2\sst}_{\beta}}
            }
        }
    }{\IfEmptyTF{#2}
        {{\Rsf^{\alpha\st}_{#1}}}
        {\Rsf^{#2\st}_{#1}}}}
}

\usepackage{tocloft}
\begin{document}

\title{Projectivity of good moduli spaces\\ of vector bundles on stacky curves}
\author{Chiara Damiolini, Victoria Hoskins, Svetlana Makarova, Lisanne Taams}
\date{}
\maketitle

\begin{abstract}
Moduli of vector bundles on stacky curves behave similarly to moduli of vector bundles on curves, except there are additional numerical invariants giving many different notions of stability.
We apply the existence criterion for good moduli spaces of stacks to show that the moduli stack of semistable vector bundles on a stacky curve has a proper good moduli space. We moduli-theoretically prove that a natural determinantal line bundle on this moduli space is ample, thus proving this moduli space is projective. Our methods give effective bounds for when a power of this line bundle is basepoint-free. As a special case, we obtain new and effective constructions of moduli spaces of parabolic bundles.
\end{abstract}

\tableofcontents

\section*{Introduction} \label{sec:intro}
\addcontentsline{toc}{section}{\nameref{sec:intro}}

Many classical results concerning vector bundles on smooth projective curves can be naturally extended to vector bundles on \emph{stacky curves}, by which we mean a smooth proper tame Deligne--Mumford stack of dimension 1 over a field $k$ that contains a scheme as a dense open subset (to ensure that the generic stabiliser is trivial).
For example, Serre duality and the Riemann--Roch Theorem extend to stacky curves, appropriate powers of the canonical sheaf give embeddings of stacky curves into weighted projective stacks and the uniformisation theorem extends to analytic stacky curves; see \cite{behrend-noohi:2006,voight-ZB:2022} and the comprehensive treatment in the upcoming thesis of the fourth author \cite{Lisanne}.

The category of coherent sheaves on a stacky curve is also of homological dimension $1$ and has a corresponding smooth Artin stack whose dimension is computed using an Euler pairing.
As our stacky curves are tame, the coarse moduli space $\pi \colon  \Cc \rightarrow C$ of the stacky curve $\Cc$ is actually a good moduli space and the stabiliser groups are finite cyclic groups.
A nice new feature of the theory of vector bundles on stacky curves is that it mixes in the representation theory of finite cyclic groups: the restriction of a vector bundle $E \rightarrow \Cc$ to a stacky point $p$ (i.e.\ a point with non-trivial stabiliser $\mu_{e}$) is a vector bundle on $B\mu_{e}$, and thus a $\ZZ/e\ZZ$-graded vector space.
This means that beyond the rank and degree, there are additional numerical invariants for vector bundles on $\Cc$ called \emph{multiplicities}, which are the dimensions of the graded pieces of the fibre at each stacky point $p$. In particular, the numerical Grothendieck group of $\Cc$ is much bigger than that of $C$. 

This larger numerical Grothendieck group leads to various different notions of stability for vector bundles on $\Cc$, unlike the classical case for $C$, where there is only one notion of slope stability. 
This abundance of notions of stability means we should be able to construct many different moduli spaces of semistable vector bundles that are birational to each other and should be related via wall-crossings. 
As the notion of stability changes, the stability of a vector bundle $E$ on $\Cc$ should be viewed as interpolating between the stability properties of the corresponding bundle $\pi_*E$ on $C$ and the properties of the graded vector spaces at the stacky points. 
For a numerical invariant $\alphab$, one can naturally define an $\alphab$-slope which is given by the Euler pairing over the rank: $\mu_\alphab(E) = \langle \alphab, E \rangle /\rk E $ (modulo a shift by a constant).
Then $\alphab$-semistability is given by checking an inequality of $\alphab$-slopes for all subbundles.

Vector bundles on stacky curves are very closely related to parabolic bundles on curves (which are vector bundles with prescribed flags in fibres over certain parabolic points). Mehta and Seshadri \cite{mehta-seshadri:1980} constructed moduli spaces of parabolic vector bundles using geometric invariant theory, where different choices of linearisation in this construction give rise to different notions of stability depending on a notion of parabolic weights, with a rich theory of variation of stability. A generalisation of this approach to generically split parahoric bundles has been carried out by Balaji and Seshadri in \cite{balaji.seshadri:2015}. 
We refer to \cref{rmk relation with parabolic VBs} for a more detailed discussion; in particular, our methods give a new construction of moduli spaces of parabolic vector bundles.

In this paper, we construct good moduli spaces of $\alphab$-semistable vector bundles on a stacky curve by applying the recent existence criteria of Alper, Halpern-Leistner and Heinloth \cite{AHLH:existence}. Using a Langton-type argument, it is easy to verify that this good moduli space is proper. The heart of this paper is showing that there is an ample determinantal line bundle on the good moduli space to prove it is projective, and in particular a scheme. We give intrinsic, moduli-theoretic ways to construct lots of sections of this determinantal line bundle by adapting ideas of Faltings \cite{faltings:1993} and others (see the discussion of related works below). Our central motivation for doing this is that although GIT automatically proves the existence of projective moduli spaces, the projective embedding is not at all explicit, as computing rings of invariants is in general extremely tricky. In fact, often the only reason one can give moduli-theoretic interpretations of GIT semistability is via the Hilbert--Mumford criterion. However, not knowing the ring of invariants, means we can say relatively little about the line bundle we obtain on the GIT quotient or which power of it yields a projective embedding. Our intrinsic approach will provide an explicit description of the sections of this line bundle and its powers, moreover we give effective bounds for a power of the determinantal line bundle to be base point free and to define a finite map to a projective space.

\subsection*{Our results}

Let $\Mstackss$ denote the stack of $\alphab$-semistable vector bundles on a smooth proper stacky curve $\Cc$ over a field $k$ with fixed numerical invariant $\betab$.
We assume that $\alphab$ is a generating numerical invariant (see \cref{definition: generating sheaves thru local condition}), which means the multiplicities of $\alphab$ at each stacky point are positive; as the name suggests, this notion is closely related to the notion of generating sheaves (see \cite{olsson-starr:2003}).
This condition on $\alphab$ ensures that the algebraic stack $\Mstackss$ is of finite type over $k$.

Using the universal bundle $\mathcal{U}_\betab \rightarrow C \times \Bun_\betab$ on the stack $\Bun_\betab$ of all vector bundles on $\Cc$ with numerical invariant $\betab$ and any vector bundle $V$ on $\Cc$, we can construct a determinantal line bundle
\[\mathcal{L}_V \coloneqq \det \left( \Rb (\mathrm{pr}_{\Bun_\betab})_* \localHom(\mathrm{pr}^*_{\Cc} V, \Uc_\betab)  \right)^{\vee}\]
on $\Bun_\betab$.
If $\langle [V], \betab \rangle = 0$, then there is also a section $\sigma_V$ of $\mathcal{L}_V$ such that $\sigma_V(E)  \neq 0$ if and only if $\Hom(V,E) = 0$ for any $E$ in $\Bun_\beta$.
While $\mathcal{L}_V$ only depends on the class $[V]$ of $V$ in the Grothendieck group $\Knot(\Cc)$, the section \emph{does} depend on $V$, and so by varying $V$ we can produce many sections of $\mathcal{L}_{[V]}$.
By modifying $\alphab$, but without changing the notion of stability, we can assume that $\langle \alphab, \betab \rangle=0$ (see \cref{lemma how to modify to orthogonal alpha}).
Fix $[V] \in \Knot(\Cc)$ whose underlying numerical class is $\alphab$ and let $\mathcal{L}_\alphab$ denote the restriction of $\mathcal{L}_{[V]}$ to $\Mstackss$. Our first result concerns this line bundle in arbitrary characteristic.

\begin{mainthm}[Semiampleness and effective bounds]\label{ThmA}
Let $\alphab$ be a generating numerical invariant such that $\langle \alphab, \betab \rangle = 0$. Then the determinantal line bundle $\mathcal{L}_\alphab$ on the stack $\Mstackss$ is semiample. More precisely, if $m $ is a positive integer and
\[ m > (g_{\Cc} -1)(\rank \betab)^2, \]
then $\mathcal{L}_\alphab^{\otimes m}$ is basepoint-free.
\end{mainthm}

In fact, we show that $\alphab$-semistability of $E \in \Bun_\betab$ is equivalent to the existence of a vector bundle $V$ with invariant $m\alphab$ (and fixed determinant) such that $\Hom(V,E) = 0$ (see \cref{proposition:semistability equivalent to Hom-vanishing}). The trickier forward implication is proved by a dimension count similar to that of Esteves \cite{esteves:1999} and Esteves--Popa \cite{esteves-popa:2004}. \cref{ThmA} is proved in \cref{theorem:semiample}.

For our main result, we assume that $\car(k) = 0$ to apply the existence criterion of \cite{AHLH:existence} and obtain a proper good moduli space. We use the sections $\sigma_V$ to show that we can separate enough vector bundles to prove that the induced line bundle in ample. For classical (i.e.\ non-stacky) curves, this was carried out in \cite{alper-etal:2022} based on the ideas of Faltings \cite{faltings:1993}, Esteves \cite{esteves:1999} and Esteves--Popa \cite{esteves-popa:2004}. Our extension of these ideas to stacky curves actually streamlines many proofs (for example, working with $\Ext$ groups rather than cohomology, means we naturally avoid many duals appearing in the proof of \cite[Proposition 5.4]{alper-etal:2022}). Let us assume that $\Mstackss$  is non-empty for the irreducible statement in the following result.

\begin{mainthm}[Existence of projective good moduli spaces]\label{ThmB}
Assume $\car(k) = 0$ and $\alphab$ is a generating numerical invariant such that $\langle \alphab, \betab \rangle = 0$.
Then $\Mstackss$ admits a proper good moduli space  $\Mspacess$ which is irreducible and normal.
Moreover, the determinantal line bundle $\mathcal{L}_\alphab$ on the stack $\Mstackss$ descends to an ample line bundle $L_\alphab$ on $\Mspacess$.
Hence, $\Mspacess$ is a projective scheme.
\end{mainthm}

This is proved in \cref{cor:langton} and \cref{theorem: det is ample}. The only reason we assume $\car(k) = 0$ is to apply the existence criterion and we could conclude the same in positive characteristic if we knew an adequate moduli space existed.
In fact, using a GIT construction (for example as in \cite{nironi:2009}), one can show $\Mstackss$ is locally reductive and thus admits an adequate moduli space. Additionally we obtain an effective bound for which power of $L_\alpha$ gives a finite morphism to a projective space (see \cref{cor:eff bdd for finite map to Pn}).

\subsection*{Related work}\label{sec: related works}

Let us very briefly give an overview of the literature concerning moduli of vector bundles on curves, before turning to the case for stacky curves.

Moduli spaces of stable vector bundles on a curve $C$ were first constructed over the complex numbers by Seshadri \cite{seshadri:1967} using their relation with irreducible (twisted) representations of the unitary group \cite{NS:1965}. This notion of stability coincided with the one coming from Mumford's geometric invariant theory \cite{mumford:GIT}. This gave rise to a projective moduli space of semistable vector bundles on $C$, which identifies S-equivalent bundles. The projective GIT quotient naturally comes with an ample line bundle, which in this case can be interpreted by taking the determinant of cohomology. Using this perspective, Faltings \cite{faltings:1993} gave an intrinsic proof that this determinantal line bundle was ample; these ideas are nicely elaborated by Seshadri \cite{Seshadri:VB}. Esteves and Popa \cite{esteves:1999,esteves-popa:2004} gave a more modern treatment that proved semiampleness by a dimension count and separates points using these ideas and dually Ext-vanishing results, together with the study of certain Schubert varieties in Grassmannians. In \cite{alper-etal:2022}, the existence criteria and these ideas were used to give a proof of projectivity that goes beyond GIT.

Let us now turn to the case of stacky curves, which have been studied in \cite{behrend-noohi:2006,voight-ZB:2022}, and vector bundles on stacky curves and their moduli is studied in the thesis of the fourth author \cite{Lisanne}. Moduli of semistable sheaves on Deligne--Mumford stacks were studied in the preprint of Nironi \cite{nironi:2009}, where he generalised the GIT-type construction of moduli of sheaves on schemes using Quot schemes. Whilst we were completing this paper, a related paper appeared \cite{das-majumder} constructing projective moduli spaces for semistable vector bundles on \emph{wild} orbifold curves in positive characteristic. However they only consider a single specific notion of stability that is well-behaved with respect to pullback along covers. Their strategy is then to choose a nice schematic cover of their orbifold curve and to embed their moduli space into the moduli space of vector bundles on the cover, which is known to be projective. We on the other hand consider a wide range of stability conditions, which are in general not well behaved with respect to pullback along covers, so their methods do not seem to apply.

This paper fits more broadly into the \emph{beyond GIT} program developed by Alper, Halpern--Leistner and Heinloth, and there are several related papers in this direction. One major breakthrough of these ideas concerns the construction of good moduli spaces of $K$-semistable Fano varieties \cite{alper-blum-etal:2020}. There are stacks project papers giving stack-theoretic constructions and projectivity proofs for moduli spaces of stable curves \cite{cheng-etal} and moduli spaces of semistable vector bundles on a curve \cite{alper-etal:2022}. Furthermore, the first three authors together with Belmans, Franzen and Tajakka gave a beyond GIT proof of projectivity of moduli spaces of representations of an acyclic quiver \cite{BDFHMT}.

\subsection*{Acknowledgements}
The authors are grateful to Jochen Heinloth, David Rydh, Dario Wei{\ss}mann for helpful discussions; in addition, Dario pointed out the preprint \cite{das-majumder} to the authors.
The authors would like to thank the Isaac Newton Institute for Mathematical Sciences, Cambridge, for support and hospitality during the programme New equivariant methods in algebraic and differential geometry, where work on this paper was undertaken. This work was supported by EPSRC grant EP/R014604/1.
S.M. was a visiting researcher at Universit\"at Duisburg-Essen in 2023--2024 and would like to thank it and Marc Levine for the hospitality.

\section{Vector bundles on stacky curves}

In this first section we aim to collect the definitions and results concerning stacky curves, including the appropriate reformulation of the Riemann--Roch Theorem and of stability for vector bundles. We refer the reader to \cite{Lisanne} for details and proofs. Throughout, we will work over a field $k$.

\subsection{Stacky curves}

A \textbf{stacky curve} is a smooth finite type geometrically connected proper Deligne--Mumford stack $\Cc$ of dimension $1$ over a field $k$, such that there exists a scheme $X$ and an open immersion $X \rightarrow \Cc$. If $\Cc$ is a stacky curve, we denote its associated coarse moduli space by $\pi \colon \Cc \to C$, which is a morphism to a smooth projective curve $C$. If $\pi_*$ is exact on quasi-coherent sheaves, then we say that the stacky curve $\Cc$ is \textbf{tame}. In this case, $\pi \colon \Cc \to C$ is in fact a good moduli space.
We note that this is always the case if $\car{k}=0$.

We say that a closed point $p$ of a stacky curve $\Cc$ is a \textbf{stacky point} if it has a non-trivial stabiliser group $G_p \coloneq \isom{p}{p}$. 
We define the \textbf{residual gerbe} of $p$ to be the unique reduced closed substack of $\Cc$ supported on $p$ and we denote it by $\Gc_p$. The \textbf{order} of $p$, denoted $e_p$, is the unique positive integer such that there exists an isomorphism $\Gc_p \cong B\mu_{e_p}$.

Throughout we fix a tame stacky curve $\Cc$ and denote the set of stacky points of $\Cc$ by $\stackypoints$ and the coarse moduli space by $\pi \colon \Cc \to C$.

\subsection{Vector bundles on stacky curves} \label{sec:invariants-vbs}

Let $E$ be a vector bundle on $\Cc$.
At each stacky point $p$ of order $e_p$, fixing an isomorphism $\Gc_p \cong B\mu_{e_p}$, defines an the inclusion $\iota_p \colon B\mu_{e_p} \rightarrow \Cc$.
Then the restriction $\iota_p^*E$ defines a $\ZZ/e\ZZ$-graded vector space
\begin{equation} \label{eq:i*E} \iota_p^*E \cong \bigoplus_{i \in \ZZ/e\ZZ}^{e_p-1} \kappa(p)^{m_{p,i}(E)}.\end{equation}
The numbers $m_{p,i}(E)$ are called the \textbf{multiplicities at $p$} of $E$ the \textbf{multiplicity vector at $p$} is defined as\[ m_p(E) \coloneqq (m_{p, 0}(E), \cdots, m_{p, e_p - 1}(E)) \in \NN^{e_p}.\] If $E$ is fixed we will denoted these simply by $m_{p,i}$ and $m_p$.
Finally the collection of all the multiplicity vectors $m_p(E)$ for every stacky point $p$ is called the \textbf{multiplicities} of $E$ and denoted by $\m(E)$ or simply $\m$ if $E$ is understood. 
Since we can resolve any coherent sheaf $F$ by a complex of two vector bundles  $E_1 \to E_0$, we define (virtual) multiplicities for $F$ as $\m (F) = \m (E_0) - \m (E_1)$.
Note that multiplicities of torsion sheaves may be negative, unlike the definition of multiplicities for torsion sheaves in \cite{Lisanne}, which does not factor through the Grothendieck group.

\begin{lemma}
[{\cite[Cor. 1.2.12]{Lisanne}}]
\label{lem:picC}
Let $\stackypoints=\{p_1, \dots, p_n\}$ and set  $e_j:=e_{p_j}$. Write $x_i:= \pi(p_i)$ and let $\pic{\Cc}$ denote the (set-theoretic) Picard group of $\Cc$. The map
\[\pic{C} \oplus \bigoplus_{j=1}^n \ZZ x_j
    \to \pic{\Cc}, \quad {(L,x_j)} \mapsto \pi^* L \otimes \Oc_\Cc\left(\frac{1}{e_j}p_j\right)\] induces a short exact sequence of abelian groups
\[
    0 \longrightarrow \bigoplus_{j=1}^n \ZZ  (e_j x_j - \Oc_C (p_j))
    \longrightarrow \pic{C} \oplus \bigoplus_{j=1}^n \ZZ x_j
    \longrightarrow \pic{\Cc} \longrightarrow 0 .
\]
\end{lemma}

Let $\Knot(\Cc)$ and $\Knum(\Cc)$ denote the Grothendieck group and numerical Grothendieck group of $\Cc$ respectively. Note that the rank and the determinant define a map  $\Knot(\Cc) \to \ZZ \oplus \pic{\Cc}$.

Further, we observe that for a coherent sheaf $F$ the multiplicities at each stacky point always add up to the rank by construction, as the same is true for vector bundles and both the rank and multiplicities are additive in short exact sequences.

\begin{corollary}
[{\cite[Theorem 1.2.27]{Lisanne}}]
	The assignment $E \mapsto (\rk(E), \det \pi_*(E), (m_{p,i}(E)))$ induces an embedding
    \[
        \Knot(\Cc) \hookrightarrow \ZZ \oplus \pic{C} \oplus
        \bigoplus_{p \in \stackypoints} \ZZ^{e_p},
    \]
    whose image coincides with the subgroup of those $(r,L,(m_{p,i}))$ such that for each $p$, we have $r = m_{p,1} + \dots + m_{p,e_p}$, and thus there is an isomorphism
    \[
        \Knot(\Cc) \cong \ZZ \oplus \pic{C} \oplus
    	\bigoplus_{p \in \stackypoints} \ZZ^{e_p - 1}.
     \]
     Taking degrees $\deg \colon \pic{C} \rightarrow \ZZ$ we get a corresponding embedding
    \[
        \Knum(\Cc) \hookrightarrow \ZZ \oplus \ZZ \oplus
        \bigoplus_{p \in \stackypoints} \ZZ^{e_p},
    \]
    and isomorphism
    \[
        \Knum(\Cc) \cong \ZZ\oplus \ZZ \oplus
    	\bigoplus_{p \in \stackypoints} \ZZ^{e_p - 1}.
    \]
\end{corollary}

A \textbf{numerical invariant} means a class in $\Knum(\Cc)$, which we typically denote by $\alphab,\betab,$ etc. By the above isomorphism, we can write $\alphab = (\rank \alphab, \deg \pi_*\alphab, (m_{p,i}(\alphab)))$. We use the terminology \textbf{algebraic invariant} to mean a class in $\Knot(\Cc)$, which we typically write as a pair $\alphaw = (\alphab,L)$ consisting of a numerical invariant and a line bundle on $C$. We say an invariant is \textbf{effective} if there is a sheaf whose class is equal to this invariant. We say a numerical invariant $\alphab$ is \textbf{positive} if $\rank \alphab > 0$ and $m_{p,i}(\alphab) \geq 0$ for all $p$ and $i$. If an invariant is positive, then it is automatically effective and moreover can be realised as the class of a non-zero vector bundle.

Another important invariant is an analogue of the parabolic degree, which shows up in the Riemann-Roch theorem and can be viewed as a degree relative to some other vector bundle.

\begin{definition}
Let $E$ be a locally free sheaf and $F$ be a coherent sheaf on $\Cc$.
We define the \textbf{$E$-degree}
\[
\deg_E(F) \coloneqq \deg(\pi_*\localHom(E,F))) - \rank(F) \deg \pi_*\localHom(E, \Oc).
\]
\end{definition}

We now give a notion of weights, which is a repackaging of the multiplicities useful for computations with $E$-degrees. Let $E$ be a locally free sheaf on $\Cc$ with multiplicities $m_{p, i}$. We define the \textbf{weights} of $E$ to be the collection 
\[w_{p, i} = w_{p,i}(E) \coloneqq \dfrac{1}{\rank E} \cdot \sum_{j = 1}^{i} m_{p,j}(E),\] where $i$ runs from $0$ to $e_p - 1$ and $p \in \stackypoints$. In particular, $w_{p,0}=0$. 

\begin{proposition}[{\cite[Theorem 1.3.19]{Lisanne}}]
\label{theorem: Edegreeformula}
Let $E$ be a locally free sheaf with weights $w_{p, i}$ and let $F$ be a locally free sheaf with multiplicities $m_{p, i}$. We have
\[
\deg_E(F) = \rank(E) \cdot \deg(\pi_*F) + \rank(E) \cdot \sum_p \sum_{i = 1}^{e_p-1}  w_{p, i}(E) \cdot m_{p, i}(F).
\]
\end{proposition}

In what follows we will use the \textbf{Euler pairing} and \textbf{Euler form}
\[ \langle E,F \rangle  \coloneqq \littleext{0}{E}{F} - \littleext{1}{E}{F} \qquad \text{and} \qquad \chi(F) = \langle \Oc_\Cc, F\rangle.
\]

\begin{theorem}[Stacky Riemann-Roch, {{\cite[Theorem 1.3.20]{Lisanne}}}]
\label{theorem: stacky RR}
Let $E$ be a locally free sheaf and $F$ be a coherent sheaf on $\Cc$. We have
\[\langle E,F \rangle = \deg_E(F) + \rank(F) \cdot \langle E,\Oc_\Cc \rangle.
\]
\end{theorem}

This shows that the pairing $\langle \, , \, \rangle$ descends to a pairing on $\Knum(\Cc)$. In what follows we will thus apply the pairing not only to sheaves, but also to elements of $\Knum(\Cc)$ or $\Knot(\Cc)$. For example, we will use notations such as $\langle E, \betaw \rangle$ or $\langle \alphaw, \betab \rangle$ or $\langle \alphab, \betab \rangle$.

\begin{remark}
Since the functor $\pi_*$ is exact,  the function $\deg_E$ descends to $\mathrm K_0(\Cc) \to \ZZ$.
	It follows from \cref{theorem: stacky RR} that $\deg_E$ actually descends to $\Knum (\Cc)$, and  we will then also write $\deg_E(\alphab)$ for $\alphab \in \Knum (\Cc)$.
\end{remark}

The \textbf{canonical sheaf} $\omega_\Cc$ on $\Cc$ is related to the pullback of the canonical sheaf $\omega_C$ on the coarse moduli space $C$ by the formula
\[
\omega_\Cc \cong \pi^*\omega_C \otimes
\bigotimes_{p \in \stackypoints} \Oc
\left( \frac{1}{e_p}p \right)^{\otimes e_p - 1}.
\]
We define the \textbf{genus} $g_\Cc$ of $\Cc$ by \[g_{\Cc} = g_C + \frac{1}{2}\sum_{p \in \stackypoints} \frac{e_p - 1}{e_p}.\]

\begin{theorem}[Serre Duality, {\cite[Theorem 1.3.7]{Lisanne}}]
\label{theorem: Serre duality}
Let $E$ be a coherent sheaf on a projective stacky curve $\Cc$, we have natural isomorphisms
\[
\hom{E}{\omega_\Cc} \cong \ext{1}{\Oc_\Cc}{E}^\vee \quad  \text{ and }  \quad  \ext{1}{E}{\omega_\Cc} \cong \hom{\Oc_\Cc}{E}^\vee.\]
\end{theorem}

In what follows, we denote by $\omegab$ the class of $\omega_\Cc$ in $\Knum (\Cc)$ and set
 \[\SD (F) \coloneqq \localHom (F,\omega_\Cc)\] the Serre dual of $F$. If $[F]=\alphab \in \Knum(\Cc)$, then $\SD(F)$ has class $\SD(\alpha)=\alphab^\vee \otimes \omegab \in \Knum(\Cc)$.
	We can thus rephrase \cref{theorem: Serre duality} as
	\begin{equation} \label{eq: pairing and SD} \langle \alphab, \betab \rangle 	= - \langle \betab , \alphab \otimes \omegab \rangle 	= - \langle \betab , \SD(\alphab^\vee) \rangle .
	\end{equation}

\begin{proposition}\label{proposition:bound for Euler pairing}
    Let $\alphab$ be a positive invariant, then
    \[ (g_C - 1) \rank(\alphab)^2 \leq - \langle \alphab, \alphab \rangle \leq (g_{\Cc} - 1) \rank(\alphab)^2. \]
    Moreover, the left hand bound is attained whenever $\alphab = [\pi^*F \otimes L]$, for some vector bundle $F$ on $C$ and a line bundle $L$ on $\Cc$. The right hand bound is attained whenver the multiplicities are balanced, i.e. $m_{p, i}(\alphab) = m_{p, j}(\alphab)$ for all $p$ and $j$.
\end{proposition}
\begin{proof}
As $\alphab$ is positive, we can choose a representative $F = \bigoplus_i \Lc_i$ of the class $\alphab$, which is a sum of line bundles. As in \cref{lem:picC}, we can write $\Lc_i \simeq \pi^*L_i \otimes \Oc(\sum_{p} \frac{a_{p,i}}{e_p} e_p)$. Then we calculate
\[
\langle \Lc_i, \Lc_j \rangle = \left\langle \Oc_{\Cc}, \pi^*(L_j \otimes L_i^\vee) \otimes \left(\sum_{p} \frac{a_{p,j} - a_{p, i}}{e_p} e_p\right)\right\rangle = \langle \Oc_C, L_j \otimes L_i^\vee\rangle - \sum_{p \: : \: a_{p, j} < a_{p, i}} 1.
\]
We have $m_{p,\ell}(\Lc_i) = 1$ if and only if $a_{p,i} = \ell$ and otherwise $m_{p,\ell}(\Lc_i) = 0$. Thus $m_{p,\ell}(F)$ counts the number of $a_{p,i}$'s such that $a_{p,i} = \ell$. Using the fact that the Euler pairing is additive, we obtain 
\[
-\langle F, F\rangle = (g_C - 1)\rank(F)^2 + \sum_p \sum_{j = 0}^{e_p - 2}\left( m_{p, j}(F) \cdot \sum_{i = j+1}^{e_p - 1} m_{p, i}(F)\right).
\]
The result now follows from the following combinatorial statement. Let $n = \sum_{i = 0}^{e - 1} m_i$ be a partition of $n$ into $e$ terms. Then 
\[
S \coloneq \sum_{i = 0}^{e - 2}\left( m_i \cdot \sum_{j = i + 1}^e m_{j}\right) \leq \frac{e-1}{2 e} n^2,
\]
moreover the bound is attained precisely when $m_i = n/e$ for every $i$. To see this note that
\[
2S + \sum_{i = 0}^{e-1} m_i^2 = \left(\sum m_i \right)^2 = n^2,
\]
so $S$ is maximal when $\sum m_i^2$ is minimal. This happens when $m_i = \frac{n}{e}$, in which case we find $2S + e \frac{n}{e}^2 = n^2$ or $S = \frac{e - 1}{2e}n^2$.
\end{proof}

\subsection{Generating sheaves and semistability}

We will now introduce the notion of generating sheaves on a stacky curve, which enhances the notion of polarisation to the stacky case.
This notion allows one to adapt the classical quot scheme arguments to construction of moduli spaces of sheaves on Deligne-Mumford stacks.
The definition that we now give is usually called the ``local condition of generation'' in \cite{olsson-starr:arxiv}, but it is easier to check in examples.

\begin{definition}
\label{definition: generating sheaves thru local condition}
	We say that a locally free sheaf $E$ is a \textbf{generating sheaf} if
	$m_{p, i} \neq 0$ for every $p \in \stackypoints$ and $0 \leq i \leq e_p - 1$.
	We say that $\alphab \in \Knum (\Cc)$ is a \textbf{generating numerical invariant} if for all stacky points $p \in \stackypoints$ and $0 \leq i \leq e_p-1$, we have $m_{p,i} (\alphab) > 0$; equivalently, $\alphab = [E]$ for some generating sheaf $E$.
\end{definition}

By \cite[Theorem 1.3.13]{Lisanne} (see also \cite[Theorem 5.2]{olsson-starr:arxiv}), a locally free sheaf $E$ is generating if and only if for all coherent sheaves $F$, the natural map $\pi^*\pi_*(\localHom(E,F)) \otimes E \rightarrow F$ is surjective.

\begin{definition}
\label{definition: E-semistability}
Let $E$ be a generating sheaf on $\Cc$ and $F$ be a coherent sheaf on $\Cc$.
We say that $F$ is $E$-\textbf{(semi)stable} if for every proper subsheaf we have \[	\mu_E(F')\coloneqq \frac{\deg_E(F')}{\rank(F')} \underset{(\leq)}{<} \frac{\deg_E(F)}{\rank(F)} = \mu_E(F).\] When $[E]=\alphab \in \Knum(\Cc)$, we will more often use the terminology $\alphab$-(semi)stable. Moreover, we note that $\mu_E(F)$ only depend on the classes $[E]=\alphab$ and $[F]=\betab$ in $\Knum(\Cc)$ and thus we will often write $\mu_\alphab(\beta)$ in place of $\mu_E(F)$.
\end{definition}

\begin{remark}
\label{remark: slope from pairing}
	Using Riemann--Roch (\cref{theorem: stacky RR}), we have
	\begin{equation}\label{equation:slope and euler char}
		\frac{\langle \alphab, \gammab \rangle}{\rank \gammab} = \mu_{\alphab}(\gammab) + \langle \alphab, \Oc_{\Cc} \rangle.
	\end{equation}
	In particular, if $\langle \alphab, \betab \rangle = 0$, then $\mu_{\alphab}(\gammab) \leq \mu_{\alphab}(\betab)$ is equivalent to $\langle \alphab, \gammab \rangle \leq 0$; and similarly for strict inequalities.
\end{remark}

\begin{lemma}\label{lemma how to modify to orthogonal alpha}
Let $\alphab$ be a generating numerical invariant. For any positive numerical invariant $\betab$, there is a  generating numerical invariant $\alphab'$ such that $\langle\alphab', \betab\rangle = 0$ and that the notions of (semi)stability with respect to $\alphab$ and $\alphab'$ coincide.
\end{lemma}
\begin{proof}
    Assume that $A := \langle \alphab, \betab \rangle \neq 0$, and let $\etab = [\Oc(q)] \in \Knum(\Cc)$ for a non-stacky point $q \in \Cc$. Pick $r \in 
    \ZZ$ so that $B= \langle \alphab \otimes  \etab^{\otimes r}, \betab\rangle$ has the opposite sign from $A$.
    Then it is straightforward to check that the numerical invariant 
    \[
        \alphab' := |B| \alphab + |A| \alphab \otimes  \etab^{\otimes r}
    \]
    is orthogonal to $\betab$ and additionally $ \alphab'$ is generating, as it is a positive linear combination of a generating invariant and an effective invariant.
    The equivalence of the corresponding notions of (semi)stability follows from the fact that $\deg_{\alphab'} = (|A| + |B|)\deg_\alphab$, as this degree is additive and preserved by tensoring by the line bundle $\mathcal{O}(rq)$ as it is the pullback of a line bundle on $C$ by virtue of $q$ being a non-stacky point. 
\end{proof}

\begin{remark}\label{rmk relation with parabolic VBs}
   Moduli stacks of parabolic bundles with fixed invariants that are semistable with respect to fixed parabolic weights, implicitly defined in \cite{mehta-seshadri:1980}, are isomorphic to moduli stacks of  vector bundles on a stacky curve with fixed invariants that are semistable with respect to a generating sheaf that depends on the weights. This was shown on the level of (monoidal) categories in \cite{borne:2007} when the weights are of the form $0 < \frac{1}{n} < \frac{2}{n} < \ldots < 1$. In \cite[Corollary~7.10]{nironi:2009} this categorical equivalence was extended to an isomorphism of stacks. This was generalised to arbitrary weights in \cite[Corollary~2.0.17]{Lisanne}. The relationship between parabolic bundles and orbifold bundles was also studied by Biswas \cite{biswas:1997}.
    Consequently \cref{ThmB} gives a new construction of the parabolic moduli spaces of Mehta and Seshadri.   
\end{remark}

We conclude this section describing the interplay between Serre duality and stability; this will be used in \cref{sec:ext-vanishing} to translate a Hom-vanishing statement into an Ext-vanishing statement.

\begin{proposition}
\label{theorem:SD-stability}
    Let $\alphab$ be a generating numerical invariant and let $E$ be a vector bundle on $\Cc$ such that $\langle \alphab, E \rangle =0$. Then $E$ is $\alphab$-semistable if and only if $\SD(E)$ is $\alphab^\vee$-semistable
\end{proposition}

\begin{proof}
    Assume that $E$ is $\alphab$-semistable.
    Given a subsheaf $F \subset \SD(E)$, we apply Serre duality to get the following equality:
    \[
    \langle \alphab^\vee , F \rangle = - \langle F , \alphab^\vee \otimes \omegab \rangle
    = - \langle \alphab , \SD(F) \rangle.
    \]
    But $\SD(F)$ is a quotient of the $\alphab$-semistable sheaf $E$, hence
    $- \langle \alphab , \SD(F) \rangle \leq 0$, as desired. For the converse, we just replace $E$ with $\SD(E)$ and $\alphab$ with $\SD(\alphab)$ and notice that $\SD(\SD(E)) \cong E$, hence the argument above applies.
\end{proof}

\section{Stacks of bundles on stacky curves}

In this section, we define stacks of vector bundles and coherent sheaves on stacky curves and state their main properties as described in \cite{Lisanne}. We then describe the main properties of stacks of semistable sheaves, and in particular show they are of finite type for semistability with respect to a generating numerical invariant.

\subsection{Properties of stacks of bundles and sheaves on stacky curves}

As always, let $\Cc$ be a tame stacky curve over a field $k$. We first introduce the stack of all vector bundles on $\Cc$ without any notion of semistability and state its basic properties. In fact, this is an open substack of the stack of all coherent sheaves on $\Cc$, and so we also introduce this stack of coherent sheaves, as it will play a role in proving the existence of good moduli spaces for stacks of semistable vector bundles.

	We denote by $\Coh$ the stack of coherent sheaves on $\Cc$, namely
	\[
		\Coh(S)=  \left\langle F
		\; \middle|
		\begin{array}{l}
			F \text{ is a coherent sheaf}\\ \text{on } \Cc \times S, 
			\text{flat over } S
		\end{array}
		\right\rangle .
	\]
Let $\Bun$ denote the substack of vector bundles on $\Cc$. For a fixed numerical invariant $\betab \in \Knum (\Cc)$, we define the corresponding open and closed substack $\Coh_\betab \subset \Coh$, or simply $\Coh_\betab \subset \Coh$ (resp.\ $\Bun_\betab \subset \Bun$), that consists of sheaves (resp.\ vector bundles) which fibrewise have invariant $\betab$.
If $\betaw = (\betab,L)$ denotes an algebraic invariant, then we denote the fibre of the map  $ \det \pi_* \colon \Bun_\betab \rightarrow \Pic(C)$ over $L$ by $\Bun_\betaw$, which is the stack of vector bundles with numerical invariant $\betab$ and whose pushforward has fixed determinant $L$. Note that $\Bun_\betaw$ has codimension $g_C$ inside $\Bun_\betab$.

By \cite{nironi:2009} the stack $\Coh$ is an algebraic stack locally of finite type over $k$ and it has affine diagonal.
The same holds for the open substacks $\Coh_{\betab}$, $\Bun$ and $\Bun_{\betab}$.
By {\cite[Theorem~2.0.8, Theorem~2.0.10]{Lisanne}}, $\Coh_\betab$ is smooth and irreducible, and thus connected. The proofs of these claims follow an inductive approach similar to \cite[Appendix A]{hoffmann:2010}; however the base case of the induction in the stacky case is nontrivial due to the more complicated geometry of stacks of torsion sheaves. Moreover $\Bun_{\betab}$ has dimension $- \langle \betab , \betab \rangle$, where this Euler pairing can be explicitly computed in terms of the rank, degree and multiplicities of $\betab$ as well as the genus, as in the proof of \cref{proposition:bound for Euler pairing}.

We note that the stack $\Bun_{\betab}$ of vector bundles on $\Cc$ with invariant $\betab$ is an iterated flag bundle over the stack of vector bundles on the coarse moduli space $C$ of $\Cc$ with invariants $\pi_*(\betab)$; however the same is not true for the stack $\Coh_{\betab}$.

\subsection{Stacks of semistable vector bundles on stacky curves}
\label{sec:Mstack-semistable-properties}

We define $\Mstackss:=\Mstacka_{\betab}$ to be the substack of $\Bun_\beta$ parametrising $\alphab$-semistable vector bundles with invariant $\betab$.

\begin{proposition}
\label{proposition: basic properties of Mstackss}
    The stack $\Mstackss$ is open inside $\Mstack$, hence it is algebraic, locally of finite type over $k$, smooth with affine diagonal, and is irreducible if it is non-empty.
\end{proposition}

The fact that semistability is an open property follows from standard arguments (for example, see \cite[Proposition 2.3.1]{huybrechts-lehn:2010} or \cite[Proposition 3.2.11]{alper-etal:2022}) applied to the case of stacky curves. This is pointed out in {\cite[Proposition 4.15, Corollary 4.16]{nironi:2009}}, but it is also claimed that stability is open, which does not hold in general if $k$ is not algebraically closed as explained in the following remark; however geometric stability is open.

\begin{remark}
\label{remark: example of stability not open}
Over the non-algebraically closed field $\RR$, stability is not an open condition, even for non-stacky curves.
Consider the family of curves $C \subset \PP^2_{\RR} \times (\AA^1_{\RR} \setminus \{0\})$ defined by the equation $x^2 + y^2 = t z^2 $, where $t$ is the coordinate on $\AA^1$.
	Let $H$ be the class of a hyperplane section of $\PP^2$
	and consider the vector bundle $E$ on $C$ defined as the non-trivial extension
	$0 \to \Oc(-H) \to E \to \Oc \to 0$.
	When $t$ is not a positive real number, the bundle $E_t$ is stable.
	However, when $t>0$, we have $C_t \cong \PP^1_{\RR}$ and $E_t = \Oc(-1) \oplus \Oc(-1)$.
	We now notice that the set $\{ t> 0\} \subset \AA^1_{\RR}$ is not constructible, and in particular, not open.
\end{remark}

We now wish to prove that the stack of semistable bundles with fixed invariant is of finite type.
Before we can do that, we cite a preliminary result about quot schemes for Deligne--Mumford stacks and prove that semistable vector bundles on stacky curves can be expressed as quotients.

\begin{theorem}{\cite[Theorem 1.5]{olsson-starr:2003}}
\label{theorem: OS quot stacks are projective varieties}
	Let $\Xc$ be a tame separated Deligne-Mumford stack of finite type over $k$. Assume that $\Xc$ is a global quotient and that its coarse moduli space $X$ is a projective variety.
	Let $E$ be a quasi-coherent sheaf on $\Xc$.
	Define the quot stack $Q$ to be the stack whose fibre over a base $B$ are groupoids of locally finitely presented quotients of $E$ that are flat and with proper support over $B$.
	Then the connected components of $Q$ are projective varieties.
\end{theorem}

\begin{lemma}
\label{lemma: preliminaries for boundedness of semistability}
	Let $\alphab \in \Knum (\Cc)$ be a generating numerical invariant, represented by a generating sheaf $E$. Let $\betab \in \Knum(\Cc)$ be another invariant.
	\begin{enumerate}[label=(\roman*)]
		\item
		\label{item: slopes of subs are bounded}
		For any sheaf $F$ on $\Cc$, there exists $\mu_{\alphab,\max}(F)$ such that for all subbundles
		$F' \subset F$, we have $\mu_\alphab (F') \leq \mu_{\alphab,\max}(F)$.
		\item
		\label{item: Ext1-vanishing for high enough slopes}
    	If $F$ is an $\alphab$-semistable sheaf with $\mu_\alphab (F) > \mu_{\alphab,\max}(E \otimes \omega_\Cc)$,
		then $\Ext^1 (E , F) = 0$.
		\item
		\label{item: global generation for high enough slopes}
		If $F$ is an $\alphab$-semistable sheaf with $\mu_\alphab (F) > \mu_{\alphab,\max}(E \otimes \omega_\Cc) + \frac{\rank (E)}{\rank (F)}$, then the map $\ev \colon \Hom(E , F) \otimes E \to F$ is surjective.
	\end{enumerate}
\end{lemma}

\begin{proof}
	Part \ref{item: slopes of subs are bounded} follows as $F'$ is a subobject of $F$, so the degree and multiplicities of $F'$ are bounded above, and the rank is non-negative.

	We now prove part \ref{item: Ext1-vanishing for high enough slopes}.
    By semistability of $F$ and the assumption on the slopes, it follows that $\Hom(F, E \otimes \omega_{\Cc}) = 0$.
    Then Serre duality implies that $\Ext^1 (E, F) = 0$.

	In order to prove part \ref{item: global generation for high enough slopes}, we will adapt a classical argument (for example, see \cite[Chapter 5]{newstead:tata}).  For any point $x \in \Cc$, let $e_x$ be the order of $x$ (which will be equal to 1 if $x$ is chosen to be non-stacky). Note that tensoring by $\Oc(-x)$ doesn't change the multiplicities, so $F(-x)$ is still semistable and $\mu_E (F(-x)) = \mu_E(F) - \frac{\rank (E)}{\rank (F)}$, hence by part \ref{item: Ext1-vanishing for high enough slopes}, we have $\Ext^1 (E, F(-x)) = 0$.
    The long exact sequence obtained from
    \[
    0 \longrightarrow F(-x) \longrightarrow F \left( -\frac{1}{e_x} x \right) \longrightarrow T \longrightarrow 0
    \]
    by applying $\Hom(E, \blank)$, where $T$ is the quotient torsion sheaf,
    implies that $\Ext^1 (E, F( -\frac{1}{e_x} x ))$ admits a surjection from $\Ext^1 (E, F(-x))$, hence it also vanishes.

	Let $\iota_x \colon \Gc_x \to \Cc$ denote the inclusion of the residual gerbe (possibly $\Gc_x = \Spec k$), and set $F_{\Gc_x} \coloneqq \iota_{x*} \iota_x^* F$.
	Applying $\Hom (E, \blank)$ to the short exact sequence

	\[
	0 \longrightarrow F \left( - \frac{1}{e_x} x \right) \longrightarrow F \longrightarrow F_{ \Gc_x } \longrightarrow 0
	\]
	yields an exact sequence
	\[
	\Hom(E, F) \longrightarrow \Hom (E, F_{ \Gc_x }) \longrightarrow \Ext^1\left( E, F \left( - \frac{1}{e_x} x \right) \right) .
	\]
	We have already proved that $\Ext^1(E, F(- \frac{1}{e_x} x)) = 0$, hence we have a surjection
	\[
	f \colon \Hom(E, F) \longrightarrow \Hom(E,F_{\Gc_x}) .
	\]
	We claim that the morphism obtained by adjunction is surjective as well:
	\[
	\ev_x \colon \Hom(E, F) \otimes E \longrightarrow F_{\Gc_x}  .
	\]
	Indeed, pick any vector $v \in F_{\Gc_x}$.
	By adjunction, we have
	\[
	\Hom_\Cc (E, F_x) = \Hom_{\mu_{e_x}} (\iota_x^* E , \iota_x^* F),
	\]
	and since $E$ is generating, there is a morphism of $\ZZ/e\ZZ$-graded vector spaces $g \colon i_x^* E \to i_x^* F$ such that $v = g(w)$ for some section $w$ in a neighborhood of $x$.
	Since $f$ is surjective, there is a morphism $h \colon E \to F$ such that $f(h) = g$.
	But now we observe that $v = \ev_x (h \otimes w)$, and we conclude that $\ev$ is surjective.
\end{proof}

\begin{proposition}
\label{proposition: Mstackss of stacky curve is finite type}
	If $\alphab$ is a generating numerical invariant, then $\Mstackss$ is of finite type.
\end{proposition}

\begin{proof}
Fix an ample line bundle $\Oc_C(1)$ on the good moduli space $\pi \colon \Cc \to C$, and for an arbitrary sheaf $F$ on $\Cc$, denote by $F(n)$ the twist $F \otimes \pi^* \Oc_C(n)$.
	Pick a generating bundle $E$ of class $\alphab$.
	For a large enough $m \in \ZZ$, we have that
	\[
		\mu_E (F(m)(-x)) > \mu_{E,\max}(E \otimes \omega_\Cc)
	\]
	for every $F \in \Mstackss(k)$ and $x \in \Cc$. Therefore, by \cref{lemma: preliminaries for boundedness of semistability}, part \ref{item: global generation for high enough slopes}, we have that for every $F \in \Mstackss$, the following map is surjective:
	\[
	\Hom (E , F(m)) \otimes E \to F(m) .
	\]
	Combining the inequality
	\[
		\mu_E (F(m)) = \mu_E (F(m)(-x)) + \frac{\rank (E)}{\rank (F)} > \mu_{E,\max}(E \otimes \omega_\Cc),
	\]
	with \cref{lemma: preliminaries for boundedness of semistability}
	\ref{item: Ext1-vanishing for high enough slopes}, we deduce that the dimension of $\Hom (E , F(m))$ is independent of $F \in \Mstackss$; call this dimension $N$.
	Therefore, every $F \in \Mstackss$ can be written as a quotient
	\[
	E(-m) ^ {\oplus N} \to F ,
	\]
	or in other words, realised as an element of the quot scheme $Q$ of quotients of $E(-m) ^ {\oplus N}$ that have a fixed numerical invariant $\betab$.
	By openness of semistability (\cref{proposition: basic properties of Mstackss}), we find an open subscheme $Q^\circ \subset Q$ that surjects onto $\Mstackss$.
	Since $\Mstackss$ is connected by \cref{proposition: basic properties of Mstackss}, we find a connected component $Q'$ of $Q$ such that $Q' \cap Q^\circ$ still surjects on $\Mstackss$.
	But by \cref{theorem: OS quot stacks are projective varieties}, $Q'$ is a projective variety, and
	this means that $\Mstackss$ is bounded.
\end{proof}

\section{Existence and properties of good moduli space}
\label{sec:GMS}

In this section we apply the existence criterion of Alper, Halpern-Leistner and Heinloth \cite[Theorem A]{AHLH:existence} to prove that the stack $\Mstackss$  admits a good moduli space in the sense of Alper \cite{alper:2013}. In this section, we will assume that $\car(k)=0$, as we only apply the existence criterion in characteristic zero due to the difference in positive characteristic between linearly reductive and reductive stabilisers  (which requires a weaker notion of an adequate moduli space). In this section, $\alphab$ will denote a generating numerical invariant.

\subsection{Applying the existence theorem}
\label{sec:existence-gms}

It follows from \cref{proposition: basic properties of Mstackss} and \cref{proposition: Mstackss of stacky curve is finite type} that, if $\alphab$ is generating, then $\Mstackss$ is an algebraic stack of finite type over $k$ and with affine diagonal. Under these assumptions we are in the position to apply the following existence criterion for good moduli spaces. We will only state this criterion in characteristic zero, as we cannot verify the additional local reductivity assumption required in positive characteristic to obtain the \'{e}tale local quotient description as in \cite{alper-hall-rydh} when the stabilisers of closed points are linearly reductive (in characteristic zero, this is always the case, as S-completness implies these stabilisers are reductive).

\begin{theorem}[Existence criteria for stacks,  {\cite[Theorem A]{AHLH:existence}}]
Let $\mathfrak{X}$ be an algebraic stack of finite type over a characteristic zero field $k$ with affine diagonal. Then $\mathfrak{X}$ admits a separated good moduli space if and only if $\mathfrak{X}$ is $\Theta$-complete and S-complete.
\end{theorem}

Let us give the definitions of the completeness conditions appearing here, which are valuative criteria involving verifying codimension 2 filling conditions.

\begin{definition}
A stack $\mathfrak{X}$ is $\Theta$-complete (resp.\ S-complete) if for every DVR $R$ with uniformiser $\pi$, every morphisms from $\mathcal{T}_R \setminus \{ \bm{0} \} \rightarrow \mathfrak{X}$ extends to $\mathcal{T}_R$ where
\[ \mathcal{T}_R = \Theta_R \coloneq [ \Spec (R[s])/\mathbb{G}_m] \quad \text{(resp.\ } \quad \mathcal{T}_R = \ST_R \coloneq [\Spec (R[s,t]/\pi - st)/\mathbb{G}_m ] \: \text{)}\]
and $\mathbb{G}_m$ acts on $s$ with weight $+1$ and $t$ with weight $-1$.
\end{definition}

By definition, $\Theta_R$ is the base change of $\Theta := [ \Spec (\ZZ[s])/\mathbb{G}_m]$ to $R$. For a detailed discussion of these conditions, we refer to \cite[$\S$6.8.2]{alper-moduli}. If $\mathfrak{X}$ is a moduli stack of objects in an abelian category, morphisms $\Theta_R \setminus \{ \bm{0} \} \rightarrow \mathfrak{X}$ can be viewed as a family over $R$ with a filtration over the generic fibre $K = \mathrm{Frac}(R)$ whose associated graded object lies in $\mathfrak{X}$, and such a morphism extends to $\Theta_R$ if the filtration and associated graded object extend to the special fibre $\kappa = R/\pi$. Similarly in this abelian setting, a morphism $\ST_R \setminus \{ \bm{0} \} \rightarrow \mathfrak{X}$ can be viewed as two families over $R$ whose generic fibres are isomorphic and this extends to $\ST_R$ if the special fibres admit opposite filtrations whose associated graded objects are isomorphic.

\begin{proposition} The stack $\Mstackss$ admits a separated good moduli space $\Mspacess$.
\end{proposition}

\begin{proof} By the above existence criterion \cite[Theorem A]{AHLH:existence}, it suffices to prove that $\Mstackss$ is $\Theta$-complete and S-complete. Let $R$ be a discrete valuation ring with residue field $\kappa$, and denote by $\pi$ its uniformiser and by $K$ its fraction field.

We will prove these valuative criteria for $\Mstackss$ by first extending our map $\mathcal{T}_R \setminus \{ \bm{0} \} \rightarrow \Mstackss$ to the stack of coherent sheaves $\Coh$ on $\Cc$. If $\Ac$ is the category of quasi-coherent sheaves on $\Cc$, the stack of coherent sheaves $\Coh$ coincides with the stack  $\Mc_{\Ac}$ introduced in \cite[$\S$7]{AHLH:existence} (see Example 7.1 and Definition 7.8 in \textit{loc.\ cit.}), thus $\Coh$ is S-complete and $\Theta$-complete by Lemmas 7.16 and 7.17 in \textit{loc.\ cit}. Alternatively, one can use the properness of the Quot scheme of sheaves on $\Cc$ (see \cite[Theorem 1.1]{olsson-starr:2003}) to prove that $\Coh$ is $\Theta$-complete.

For $\Theta$-completeness, we consider the stack $\Theta_R := [ \Spec (R[s])/\mathbb{G}_m]$. We identify $\Theta_R \setminus \{ \bm{0} \}$ with $\Spec(R) \underset{\Spec(K)}{\sqcup} \Theta_K$, so a morphism $\Theta_R \setminus \{ \bm{0} \} \rightarrow \Mstackss$ corresponds to a semistable vector bundle $F$ over $\Cc_R$ with a filtration
\[0 = F_{K}^{-m} \subset \cdots \subset F_{K}^{\ell -1} \subset F_{K}^{\ell} \subset F_{K}^{\ell +1} \subset \cdots \subset F_K^{n} = F_K\]
of the generic fibre whose associated graded object $\mathrm{gr}(F_K^\bullet) = \oplus_{\ell} F_{K}^{\ell}/ F_{K}^{\ell-1}$ lies in $ \Mstackss$. In particular, we must have $\mu_{\alphab}(F_{K}^{\ell}) = \mu{\alphab}(F_K)$ and all these sheaves are $\alphab$-semistable. This morphism extends to $\Theta_R$ if the above filtration and associated graded object extends over the special fibre $\kappa$ of $R$ in $ \Mstackss$.
By the discussion above, $\Coh$ is $\Theta$-complete and so we can extend the above morphism to $\Theta_R \rightarrow \Coh$ which gives a filtration  $0 = F^{-M} \subset \cdots \subset F^{\ell -1} \subset F^{\ell} \subset F^{\ell +1} \subset \cdots \subset F^{N} = F$ of coherent sheaves on $\Cc_R$ that restricts to the above filtration of $\alphab$-semistable vector bundles over $\Cc_K$.
Since the subsheaves $F^{\ell}$ are flat over $R$, they have the same $\alphab$-slope as the generic fibre.
Hence we also have $\mu_{\alphab}(F_{\kappa}^{\ell}) = \mu_{\alphab}(F_\kappa)$ over the special fibre and deduce each $F_{\kappa}^{\ell}$ is $\alphab$-semistable from the semistability of $F_\kappa$ using that $F_{\kappa}^{\ell} \subset F_\kappa$ have the same $\alphab$-slope.
As the category of $\alphab$-semistable vector bundles of fixed slope is abelian with the same cokernels and kernels as $\Ac$ (for example, see the appendix of Nori in \cite{Seshadri:VB}), we deduce that $\mathrm{gr}(F^\bullet_\kappa)$ is also $\alphab$-semistable. This proves the image of $\phi$ is contained in $\Mstackss$.

For S-completeness, we consider the stack $\ST_R:= [\Spec(R[s,t]/\pi - st)/\Gm]$. We identify $\ST_R\setminus \{ \bm{0} \}$ with $\Spec(R) \underset{\Spec(K)}{\sqcup}\Spec(R)$, so a morphism $\ST_R \setminus \{ \bm{0} \} \rightarrow \Mstackss$ corresponds to two semistable vector bundles $F_{-\infty}$ and $F_{\infty}$ over $\Cc_R$ with a fixed isomorphism over $\Cc_K$. This extends to $\ST_R$ if
we can find a system of vector bundles $(F_\ell)_{\ell \in \ZZ}$ which fit in a diagram
\begin{equation} \label{eq:S-completeness} \xymatrix{
\dots \ar@/^/[r]^{s_{\ell-2}} & F_{\ell-2} \ar@/^/[r]^{s_{\ell - 1}} \ar@/^/[l]^{t_{\ell - 3}} & F_{\ell-1} \ar@/^/[r]^{s_\ell} \ar@/^/[l]^{t_{\ell - 2}} & F_{\ell} \ar@/^/[r]^{s_{\ell + 1}} \ar@/^/[l]^{t_{\ell - 1}} & F_{\ell+1} \ar@/^/[r]^{s_{\ell + 2}} \ar@/^/[l]^{t_\ell} & F_{\ell+2} \ar@/^/[r]^{s_{\ell + 3}} \ar@/^/[l]^{t_{\ell + 1}} & \dots \ar@/^/[l]^{t_{\ell + 2}},
}
\end{equation} where  \begin{enumerate}[label=(S\arabic*)]
\item \label{it:S1} the maps $s_i$ and $t_i$ are injections such that $s_i \circ t_{i - 1}$ and $t_i \circ s_{i + 1}$ are given by multiplication by $\pi$ (occasionally we will omit the subscripts and denote these maps by $s$ and $t$);
\item \label{it:S2} there exists an $N  \in \ZZ$ such that for every $n \geq N$ one has isomorphisms $F_n \cong F_\infty$ and $F_{-n} \cong F_{-\infty}$ commuting with the morphisms $s_{n + 1} \colon F_n \to F_{n+1}$ and $t_{-n - 1} \colon F_{-n} \to F_{-n-1}$, respectively;
in particular, $s_n$ and $t_{-n}$ are isomorphisms for $n > N$;
\item \label{it:S4} the map $s$ induces an injection $F_{\ell-1}/t(F_\ell) \to F_{\ell}/t(F_{\ell+1})$, and analogously the map $t$ induces an injection $F_{\ell+1}/s(F_{\ell}) \to F_{\ell}/s(F_{\ell-1})$;
\item \label{it:S5} the sheaf
\[ \mathrm{gr}(F) \coloneqq \bigoplus_{ \ell \in \ZZ} \dfrac{F_{\ell}/t(F_{\ell+1})}{s(F_{\ell-1}/t(F_\ell))} \cong  \bigoplus_{ \ell \in \ZZ} \dfrac{F_{\ell}/s(F_{\ell-1})}{t(F_{\ell+1}/s(F_{\ell}))} \cong
 \bigoplus_{ \ell \in \ZZ} \dfrac{F_{\ell}}{s(F_{\ell-1}) + t(F_{\ell+1})}
\] over $\Cc_\kappa$ is an $\alphab$-semistable vector bundle.
\end{enumerate}

Since $\Coh$ is S-complete, we can uniquely find a system of {coherent} sheaves $(F_\ell)_{\ell_\in \ZZ}$ as in \eqref{eq:S-completeness} satisfying conditions \ref{it:S1}, \ref{it:S2} and \ref{it:S4}.
Since the maps $s$ and $t$ are injective, this implies that $F_\ell$ is a vector bundle for every $\ell$ and thus we are left to show that \ref{it:S5} holds. Note that conditions \ref{it:S1}--\ref{it:S4} tell us that
\begin{equation} \label{eq:filtration}
	0 =\dfrac{F_{-\infty}}{F_{-\infty}}= \dfrac{F_{-N-1}}{t(F_{-N})}  \subset \dfrac{F_{-N}}{t(F_{-N+1})} \subset \dots \subset \dfrac{F_{N-1}}{t(F_{N})} \subset \dfrac{F_{N}}{t(F_{N+1})} = \dfrac{F_\infty}{t(F_\infty)} = {F_\infty}|_\kappa
\end{equation}
is a finite filtration of ${F_\infty}|_\kappa$ (and similarly for $F_{-\infty}|_\kappa$). Recall that  $F_\infty$ has numerical invariant $\betab$ and it is $\alphab$-semistable, with $\langle \alphab, \betab\rangle =0$. Combining semistability (as in \cref{remark: slope from pairing}) together with \eqref{eq:filtration}, we obtain
\[ 0 = \langle \alphab, F_\infty \rangle  \geq \left\langle \alphab,  \dfrac{F_{\ell-1}}{t(F_\ell)}\right\rangle = \langle \alphab, F_{\ell-1} \rangle -  \langle \alphab, t(F_\ell) \rangle = \langle \alphab, F_{\ell-1} \rangle -  \langle \alphab, F_\ell \rangle,
\] where the last equality follows from the fact that $t$ is injective.
Thus we have that $\langle \alphab, F_{\ell-1} \rangle \leq  \langle \alphab, F_\ell \rangle$.
Repeating the argument with $F_{-\infty}$ we obtain the reverse inequality $\langle \alphab, F_{\ell-1} \rangle \geq \langle \alphab, F_\ell \rangle$ which forces $\langle \alphab, F_{\ell-1} \rangle=\langle \alphab, F_{\ell} \rangle$ for every $\ell$, and thus $\langle \alphab, F_\ell \rangle =0$.
Since $F_\ell \subset F_\infty$ is a subbundle of the same $\alphab$-slope and $F_\infty$ is $\alphab$-semistable, we conclude $F_\ell$ is also $\alphab$-semistable.
Again, as the category of $\alphab$-semistable vector bundles of fixed slope is abelian, we deduce that ${F_{\ell-1}}/{t(F_\ell)}$ and
\[\mathrm{gr}(F)_\ell \coloneqq \dfrac{F_{\ell}/t(F_{\ell+1})}{s(F_{\ell-1}/t(F_\ell))}\] are $\alphab$-semistable. By semistability of $\mathrm{gr}(F)_\ell$, this sheaf is torsion free, and thus a vector bundle, which completes the proof.
\end{proof}

\begin{remark}
    This does not prove that the stack $\Bun$ is S-complete (or $\Theta$-complete); indeed it is not in general, as the cokernel of an inclusion of locally free sheaves may not be locally free (see \cite[Proposition 6.8.31 and Remark 6.8.33]{alper-moduli}).
\end{remark}

\begin{corollary}\label{cor:langton}
    The good moduli space $\Mspacess$ is a normal
    and proper algebraic space of finite type over $\Spec(k)$, which is irreducible if it is non-empty.
\end{corollary}

\begin{proof}
The stack $\Mstackss$ is irreducible and smooth (see \cref{proposition: basic properties of Mstackss}). 
	By \cite[Theorem 4.16]{alper:2013}, the irreducibility and normality of $\Mstackss$ descend to its good moduli space $\Mspacess$.
	We are left to prove properness which, in view of \cite[Theorem A]{AHLH:existence}, amounts to showing that the stack $\Mstackss$ satisfies the existence part of the valuative criterion of properness.
	For this, we can assume that $k$ is algebraically closed. For a non-stacky curve, this is a classical result of Langton \cite[Theorem at page 99]{Langton:1975} which was extended to the case of stacky curves in  \cite[Theorem 1.1]{huang:2023:langton}.
\end{proof}

The remainder of the paper is devoted to proving that the good moduli space is projective, and thus in particular is a scheme rather than just an algebraic space. Our first step is to construct the line bundle from which we will obtain a projective embedding.

\section{Determinantal line bundles}
\label{sec:det-line-bdles}

In this section we construct the determinantal line bundle $\Lc_V $ over $\Mstack$ which is naturally associated to a vector bundle $V$ on $\Cc$.
We will see that when $\langle [V] , \betab \rangle =0$, this line bundle has a global section. 
The properties of this line bundle will be crucial to proving the projectivity of $\Mspacess$ in \cref{sec:proj-of-gms}.

\subsection{Definition and main properties of determinantal line bundles}

Consider the diagram
\[  \begin{tikzcd}
	&
	\Uc_\betab  \ar[no head, dotted, d]
	\\
	V \ar[no head, dotted, d]
	&
	\Cc \times \Mstack \ar[rd, "p"] \ar[ld, "q"']
	&
	\\
	\Cc && \Mstack
\end{tikzcd}
\]
where $\Uc_\betab$ is the universal vector bundle on $\Cc \times \Mstack$ and $V$ is a vector bundle on $\Cc$. Then we define
\begin{equation}
	\label{eq: definition of determinantal line bundle}
	\Lc_V \coloneqq \det \left( \Rb p_*\localHom(q^*V, \Uc_\betab)  \right)^{\vee}
\end{equation}
and we call this bundle the \textbf{determinantal line bundle on $\Mstack$ associated to $V$}. Concretely, by base change \cite[Corollary 4.13]{hall-rydh:2017:perfect-complexes}, at a point $E \in \Mstack(k)$ the fibre is given by
\[
\Lc_V|_E = \det \Ext^0(V,E)^{\vee} \otimes \det \Ext^1(V,E).
\]
The complex $\Rb p_*\localHom(q^*V, \Uc_\betab)$ is locally represented by a complex of vector bundles $K^0 \rightarrow K^1$ on $\Mstack$. To see this, let $d \geq 0$ be an integer and consider the open substack $\Xc_d \subset \Mstack$ consisting of $F$ such that $\Ext^1(V(-d), F) = 0$. By base change, this means that the fibres of $\Rb^1 p_* \localHom (q^* V(-d) , \Uc_\beta )_{| \Xc_d}$ are zero, hence this sheaf vanishes. It is clear that these subtacks $\Xc_d$ cover $\Mstack$.
Now consider the short exact sequence of coherent sheaves on $\Cc \times \Xc_d$
\[
0 \rightarrow \localHom(q^*V, \Uc_\betab)|_{\Xc_d} \rightarrow \localHom(q^*V(-d), \Uc_\betab)|_{\Xc_d} \rightarrow Q_d \rightarrow 0.
\]
Applying $\Rb p_*$ to the short exact sequence we get a long exact sequence
\begin{align*}
0 &\rightarrow \Rb^0p_*\localHom(q^*V, \Uc_\betab)|_{\Xc_d} \rightarrow \Rb^0 p_* \localHom(q^*V(-d), \Uc_\betab)|_{\Xc_d} \rightarrow \\
 &\rightarrow \Rb^0 p_* Q_d \rightarrow \Rb^1 p_* \localHom(q^*V, \Uc_\betab)|_{\Xc_d} \rightarrow 0 \rightarrow \Rb^1 p_* Q_d \rightarrow 0.
\end{align*}

Notice that $Q_d$ is the tensor product of a vector bundle $\localHom(q^*V , \Uc_\betab)|_{\Xc_d}$ with $q^* \Oc_D(d)$, where $D$ is a divisor on $\Cc$ corresponding to the embedding $\Oc_\Cc \to \Oc_\Cc (d)$. Since $q^* \Oc_D(d)$ is flat over $\Xc_d$, it follows that $Q_d$ is too. By the cohomology and base change theorem \cite[Theorem~A]{cohomology-basechange} it follows that 
$\Rb^0 p_* \localHom(q^*V(-d), \Uc_\betab)|_{\Xc_d} = p_* \localHom(q^*V(-d), \Uc_\betab)|_{\Xc_d}$ and $\Rb^0 p_* Q_d = p_* Q_d$ and they are vector bundles. In particular, we have a quasi-isomorphism of complexes 
\[ \Rb p_*\localHom(q^*V, \Uc_\betab)|_{\Xc_d} \simeq \left[p_* \localHom(q^*V(-d), \Uc_\betab)|_{\Xc_d} \overset{\delta_d}{\longrightarrow} p_* Q_d \right] \]
where the latter is a two term complex of vector bundles.

Let $V$ be a vector bundle such that $\langle V, \betab\rangle = 0$, then from the local picture we can see that $\Lc_V$ comes with a natural section. Namely, we take $\det(\delta_d) \in H^0(\Xc_d, \Lc_V|_{\Xc_d})$, and these sections glue together to a global section $\sigma_V$. Note that on the locus $\Xc_0$ the complex is given by the unique map $\delta_0 \colon 0_{\Xc_0} \rightarrow 0_{\Xc_0}$ between the zero vector bundles, so $\Lc_V$ trivialises on $\Xc_0$ via the canonical section $\det(\delta_0) = 1$. Hence for $E \in \Mstack(k)$, we have
\begin{equation} \label{prop:sigma} \sigma_V|_E \neq 0 \, \text{ if and only if } \,  \Hom(V,E) =\Ext^1(V, E) = 0. \end{equation}

Given an exact sequence of vector bundles $0 \rightarrow V' \rightarrow V \rightarrow V'' \rightarrow 0$, we have by construction $\Lc_V \cong \Lc_{V'} \otimes \Lc_{V''}$. It follows that $\Lc_{V}$ only depends on the class $[V]$ of $V$ in the Grothendieck ring  $K_0(\Cc)$. However, the section $\sigma_V$ \emph{does} depend on $V$ and we will leverage this fact to construct many different sections of $\Lc_{[V]}$ using vector bundles $W$ with the same class as $[V]$. If $\alphaw := [V]$, then we will write $\Lc_{\alphaw}$ instead of $\Lc_{[V]}$.

Lastly, we point out that the determinantal line bundle may vary as the algebraic invariant $[V]$ varies in the same numerical invariant class, and provide a formula for the dependence.
This formula is a generalisation of \cite[Fact on p. 6]{marian-oprea:2007} (see also \cite{drezet-narasimhan:1989}), as well as a reformulation in a choice-independent way.
For this, we use the following notation: given an integer combination $Z = \sum_i x_i - \sum_j y_j$ of points on $\Cc$ and a vector bundle $\Uc$ on $\Cc \times \Mstackss$, we define the following complex of sheaves on $\Mstackss$:
\[\Uc_Z = \bigoplus_i \Uc|_{\{x_i\} \times \Mstackss} [0] \oplus \bigoplus_j \Uc|_{\{y_j\} \times \Mstackss} [1].
\]

\begin{proposition}\label{proposition:determinantal line bundles dependence on V}
	Let $V$ and $V'$ be two vector bundles on $\Cc$ with the same numerical invariants $\alphab$.
	Denote by $D$ and $D'$ the divisors of $\det \pi_* V$ and $\det \pi_* V'$, respectively,
	and view them as combinations of points of $\Cc$ via the identification $\pi \colon |\Cc| \to |C|$.
	Then we have the following isomorphism over $\Mstackss$:
	\[ \Lc_V = \Lc_{V'} \otimes \det (\Uc_{\betab, D-D'})
	. \]
\end{proposition}

The proof relies on the formula $\Lc_{\Oc_x} = \det \Uc_{\betab, x}$, which is obtained by a direct calculation using the formulas from \cref{appendix: GD}. 
Since this result will not be used in this paper, we postpone the full details of the proof to  \cref{appendix: proof of det line bundle formula}.

\section{Vanishing results}
\label{sec:vanishing}

We consider $\alphab ,\betab \in {\Knum}(\Cc)$ such that $\langle \alphab , \betab \rangle=0$. As seen in the previous section, to specify the determinantal line bundle we also need to fix a determinant. Thus we will fix an algebraic invariant $\alphaw = (\alphab,L)$ and work with the line bundle $\Lc_{\alphaw}$ on $\Mstack$ and produce sections $\sigma_V$ of $\Lc_{\alphaw}^{\otimes m}$ using vector bundles $V \in \Bun_{m\alphaw}$ with numerical invariant $m \alphab$ and fixed determinant $L^{\otimes m}$ as in the previous section. First, we show for $m \gg 0$ (and in fact we give an effective bound) and for any $\alphab$-semistable vector bundle $E$ with invariant $\betab$, we can find a vector bundle $V  \in \Bun_{m\alphaw}(k)$ such that $\Hom(V,E) = 0$, or equivalently $\sigma_V(E) \neq 0$, which will allow us to prove that the restriction of $\Lc_{\alphaw}$ to $\Mstackss$ is semiample in \cref{theorem:semiample} below. Throughout this section, we assume $k=\bar{k}$ to have the existence of $k$-points of $\Bun_{m\alphaw}$.

\subsection{Hom-vanishing}
\label{sec:hom-vanishing}

We will describe the codimension of loci where Hom-vanishing fails by using \textbf{stacks of short exact sequences} which we now introduce.
For any numerical invariants $\betab_1, \betab_2$, there is a stack $\Extstack_{\betab_2,\betab_1}$ whose objects over $S$ are
\[
	\Extstack_{\betab_2, \betab_1}(S)=  \left\langle
	0 \rightarrow E_1 \rightarrow F \rightarrow E_2 \rightarrow 0
	\; \middle|
	\begin{array}{l}
		\text{short exact sequence,} \\
		E_i \in \Mstack[\betab_i](S) 
	\end{array}
	\right\rangle ,
\]
and whose morphisms are isomorphisms of short exact sequences, i.e. triples $\psi_1 \colon E_1 \to E'_1$, $\phi \colon F \to F'$, $\psi_2 \colon E_2 \to E'_2$ that make the two squares commute.

This stack admits natural forgetful maps
\[
	\xymatrix@C=1em{  & \Extstack_{\betab_2,\betab_1} \ar[ld]_{\pi_{13}} \ar[rd]^{\pi_2} & & & &  E_1 \hookrightarrow F \twoheadrightarrow E_2 \ar@{|->}[ld] \ar@{|->}[rd] &  \\ \Mstack[\betab_2] \times \Mstack[\betab_1] & & \Mstack[\betab_2 +\betab_1] & & (E_2,E_1)  & & F.}
\]
The morphism $\pi_{13}$ is a vector bundle stack and thus smooth (\cite[Theorem 2.0.4]{Lisanne}): the fibre over $(E_2, E_1)$ is isomorphic to $[\Ext^1(E_2, E_1)/\Ext^0(E_2,E_1)]$ and so the relative dimension of $\pi_{13}$ is equal to $-\langle \betab_2, \betab_1 \rangle$.
Hence $\Extstack_{\betab_2,\betab_1}$ is smooth of dimension  $-\langle \betab_1, \betab_1 \rangle -\langle \betab_2, \betab_2 \rangle -\langle \betab_2, \betab_1 \rangle$.
The morphism $\pi_2$ is representable which can be seen in two different ways: the fibres are Quot schemes, or the corresponding functor is faithful, as a morphism of short exact sequences which is the identity on $F$ must also be the identity on $E$ and $G$.

The following proposition can be seen as an extension of \cite[Lemma 3.5.8]{alper-etal:2022}; however we actually simplify the proof by doing a dimension count on the stack of vector bundles (with fixed invariants and determinant) and using the Euler pairing to simplify computations.

\begin{proposition}
\label{proposition:Hom-vanishing}
Let $\alphaw \in \Knot (\Cc)$ be a positive algebraic invariant, and let $\betab$ be a positive numerical invariant such that $\langle \alphaw, \betab\rangle = 0$. 
Let $\etaw \in \Knot(\Cc)$ be an effective algebraic invariant. Then there exists a constant $\kappa = \kappa_{\alphaw, \betab, \etaw}$ such that for any $m > \kappa$ and any $E \in \Mstack[\betab](k)$, a general vector bundle $V \in \Mstack[m \alphaw + \etaw] (k)$ satisfies the following conditions.
\begin{enumerate}[label=(\roman*)]
\item \label{it:hom-van-1}Any non-zero morphism $f \colon V \rightarrow E$ satisfies $\langle \alphaw, \im(f)\rangle \geq 0$.
\item \label{it:hom-van-2} If we assume $E$ is $\alphaw$-stable, then every non-zero map $f \colon V \rightarrow E$ is surjective.
\item \label{it:hom-van-3} If $E$ is $\alphaw$-stable and $\langle \etaw, \betab\rangle \leq 0$, then $\Hom(V, E) = 0$.
\end{enumerate}
\end{proposition}

\begin{proof}
We first show that the general vector bundle $V \in \Mstack[m \alphaw + \etaw](k)$ one has that $\pi_*V$ is $\ell$-regular for some $\ell$ independent of $m$.
Clearly this condition is open, so we just have to show there exists such a bundle.
Let $E_1 \in \Mstack[\alphaw](k)$ be such that $\pi_*E_1$ is $\ell_1$-regular and $E_2 \in \Mstack[\alphaw + \etaw](k)$ such that $\pi_*E_2$ is $\ell_2$-regular, for some $\ell_1$ and $\ell_2$.
Then $V = E_1^{\oplus m - 1} \oplus E_2 \in \Mstack[m\alphaw + \etaw](k)$ is such that $\pi_*V$ is $\ell$-regular for $\ell:=\max(\ell_1, \ell_2)$, which is independent of $m$.

For \ref{it:hom-van-1}, we will show that the locus inside $\Mstack[m \alphaw + \etaw]$ where the desired Hom-vanishing condition fails has positive codimension.
We will stratify this locus by the possible algebraic invariants of the image of the non-zero maps $f \colon V \to E$ such that $\langle \alphaw, \im(f)\rangle < 0$ and $\pi_* V$ is $\ell$-regular.
Let $\gammab$ be the numerical invariant of $G \coloneqq \im(f)$, 
and recall that $\alphab, \etab$ are the numerical invariants of $\alphaw, \etaw$, respectively.
We claim that there are only finitely many values of $\gammab$ that can appear.
Since $G \subset E$, we have $1 \leq \rank(\gammab) \leq \rank(\betab)$ and the multiplicities of $G$ are bounded by those of $E$, so it remains to bound the degree of $G$.
In fact, we will bound the $\alphab$-degree of $G$ and see that our bounds are independent of $m$.
Since $\pi_*V$ is $\ell$-regular, $\pi_* G$ is $\ell$-regular as well, and we have $\deg \pi_* G(\ell) \geq 0$, so $\deg(\pi_* G) \geq -\ell \rank G$.
On the other hand, by our assumption on $f$, we have $\deg_\alphab \gammab + \rank(\gammab) \cdot \langle \alphab , \Oc \rangle = \langle \alphab, \gammab\rangle < 0$, so combining this with \cref{theorem: Edegreeformula} as well as the inequality we already obtained from $\ell$-regularity, we get
\begin{equation}\label{bound on deg gamma}
-  \ell \rank (\gammab) \leq \deg \pi_*\gammab \leq \deg_\alphab \gammab < -\rank (\gammab) \cdot \langle \alphab, \Oc\rangle.
\end{equation}
Hence there are finitely many possibilities for $\gammab$.
For each of these finitely many $\gammab$ with $\langle \alphab,\gammab \rangle < 0$, we let $B_\gammab$ be the locus of $V \in \Mstack[m \alphaw + \etaw]$ where there is a non-zero morphism $f \colon V \rightarrow E$ whose image $G$ has invariant $\gammab$.

Now consider the diagram of vector bundles, whose invariants are displayed nearby in blue.
\begin{equation*}
\begin{tikzcd}[column sep={20mm,between origins}]
 & |[font=\small]| \textcolor{blue!50!cyan!60!black}{m\alphab +\etab- \gammab} & |[font=\small]|  \textcolor{blue!50!cyan!60!black}{m\alphab +\etab} &  |[font=\small]|  \textcolor{blue!50!cyan!60!black}{\gammab} & \\ [-7mm]
    0 \ar[r] & K \ar[r] & V \ar[r] \ar[rd, "f"'] & G \ar[r] \ar[d, hook]& 0 \\
    &&& E\\ [-7mm]     &&& |[font=\small]|  \textcolor{blue!50!cyan!60!black}{\betab}
\end{tikzcd}
\end{equation*}
Let $\Ec_\gammab$ be the substack of  $\Extstack_{\gammab , m\alphab +\etab- \gammab}$, given by short exact sequences where the determinant of the pushforward of the middle term is $\det \pi_* (m\alphaw + \etaw)$.
Notice that $\Ec_\gammab$ is the pullback of $\Extstack_{\gammab , m\alphab +\etab- \gammab} \to \Pic(C)$, sending $(E_1 \to F \to E_2) \mapsto \det \pi_* F$,
along $\det \pi_* (m \alphaw + \etaw ) : B\Gm \hookrightarrow \Pic (C)$;
hence $\dim \Ec_\gammab = \dim \Extstack_{\gammab , m\alphab +\etab- \gammab} - g_C$.
Since the middle projection $\mathcal{E}_\gammab \rightarrow  \Mstack[m\alphaw + \etaw]$ is representable and its image contains the locus $B_\gammab$, it follows that 
{\small
\begin{align*}
\operatorname{codim} B_{\gammab} &\geq \dim \Mstack[m\alphaw + \etaw] - \dim \Ec_\gammab\\
&= -\langle m\alphab + \etab, m\alphab + \etab\rangle - g_C - \left(\dim \Mstack[m\alphab + \etab - \gammab] + \dim \Mstack[\gammab] - \langle \gammab, m\alphab + \etab - \gammab \rangle - g_C\right)\\
&= -\langle m\alphab + \etab, m\alphab + \etab\rangle + \langle m\alphab + \etab - \gammab, m\alphab + \etab - \gammab\rangle + \langle \gammab, \gammab\rangle + \langle \gammab, m\alphab + \etab - \gammab \rangle \\
&= \langle \gammab, \gammab \rangle - m\langle \alphab, \gammab\rangle - \langle \etab, \gammab\rangle. \addtocounter{equation}{1}\tag{\theequation} \label{codimension}
\end{align*}
}%
Since $\langle\alphab, \gammab\rangle < 0$ by assumption, this codimension is positive for sufficiently large $m$, namely for 
$m > \frac{\langle \gammab - \etab, \gammab \rangle}{\langle \alphab , \gammab \rangle}$.

For statement \ref{it:hom-van-2}, if $\im(f)$ is a proper subbundle, then by $\alphab$-stability of $E$ we conclude $\langle \alphab, \im(f)\rangle < 0$ as in \cref{remark: slope from pairing},
which contradicts \ref{it:hom-van-1}. Hence $\im(f)$ is either $0$ or $E$.

Finally to prove statement \ref{it:hom-van-3}, let $f \colon V \rightarrow E$ be a non-zero map.
By \ref{it:hom-van-2} we may assume that $f$ is surjective, so we get an exact sequence
\begin{equation*} \label{eq:diagram-F} \begin{tikzcd}[column sep={20mm,between origins}]
 & |[font=\small]| \textcolor{blue!50!cyan!60!black}{m\alphab +\etab - \betab} & |[font=\small]|  \textcolor{blue!50!cyan!60!black}{m\alphab +\etab} &  |[font=\small]|  \textcolor{blue!50!cyan!60!black}{\betab} & \\ [-7mm]
    0 \ar[r] & K \ar[r] & V \ar[r, "a"] & E \ar[r] & 0,
\end{tikzcd}
\end{equation*}
where now the right side is our fixed bundle $E$.
Let $\Ec$ be the substack of $\Extstack_{\betab, m\alphab +\etab - \betab}$ where the final term is abstractly isomorphic to $E$ and the determinant of the pushfoward of the middle term is $\det \pi_*(m\alphaw + \etaw)$.
Then $\Ec$ has dimension $\dim \Mstack[m\alphab + \etab - \betab] - \langle \betab, m\alphab + \etab - \betab \rangle - \dim \Aut(E) - g$, and the morphism $\mathcal{E} \rightarrow \Mstack[m\alphaw + \etaw]$ is representable and its image contains the locus $B \subset \Mstack[m\alphaw + \etaw]$ consisting of $V$ such that $\Hom(V,E) \neq 0$.
Hence, using that $\dim \Aut(E) \geq 1$,
{\small
\begin{align*}
\operatorname{codim} B &\geq \dim \Mstack[m\alphaw + \etaw] - \dim \Ec \\
&\geq -\langle m\alphab + \etab, m\alphab + \etab\rangle - g_C + \langle m\alphab + \etab - \betab, m\alphab + \etab - \betab\rangle + \langle \betab, m\alphab + \etab - \betab \rangle + 1 + g_C\\
&= -\langle m\alphab + \etab, \betab\rangle + 1
= - \langle \etab, \betab\rangle + 1,
\end{align*}
}%
which is positive precisely when $\langle \etab, \betab\rangle \leq 0$.
\end{proof}

To prove semiampleness, we only need to apply the above proposition to the case $\etaw = 0$.
However in \cref{sec:ext-vanishing}, in order to be able to separate points using the sections $\sigma_V$, we will use this proposition when $\etaw = \pm \deltaw$, where $\deltaw$ is the algebraic invariant of a degree $1$ torsion sheaf supported at a non-stacky point

\begin{proposition}\label{prop:eff bdd semiample}
Consider the situation of  \cref{proposition:Hom-vanishing}, and assume in addition that $\langle \etab, \gammab \rangle \leq 0$ for every positive numerical invariant $\gammab$.
Then the constant $\kappa$ from \cref{proposition:Hom-vanishing} can be chosen to be 
\[
\kappa_\betab \coloneq \max((g_\Cc - 1)(\rank \betab)^2, 0).
\]
Note that the condition on $\etab$ is satisfied for $\etab = 0$ and $\etab = \deltab$ is the numerical class of a skyscraper sheaf at a non-stacky point.
When $\etab = -\deltab$, we can choose the bound
\[
\kappa_\betab^+ \coloneq \max\left(\left(g_\Cc - 1 + \frac{1}{\rank{\betab}}\right)(\rank \betab)^2, 0\right).
\]

\end{proposition} 

\begin{proof}
We need to ensure that the quantity of \eqref{codimension} is positive. It suffices to take \[ \kappa \geq \max_\gammab\left(\frac{\langle \gammab, \gammab\rangle - \langle \etab, \gammab\rangle}{\langle\alphab, \gammab\rangle}\right),\] where $\gammab$ runs over the finite list of numerical invariants $\gammab$ of subbundles of $E$, satisfying $\langle \alphab, \gammab\rangle < 0$. Since $\langle \alphab, \gammab\rangle \leq -1$ and by assumption $\langle \etab, \gammab\rangle \leq 0$, it suffices to take 
\[
\kappa \geq \max_\gammab (- \langle \gammab, \gammab\rangle),
\]
where $\gammab$ runs over the invariants of subbundles of $E$. By \cref{proposition:bound for Euler pairing}, we have
\[
-\langle \gammab, \gammab\rangle \leq (g_\Cc - 1)(\rank \gammab)^2
\]
and as we need $\kappa \geq 0$, we conclude the claimed bound. For the case $\etab = -\deltab$ we follow the same argument, but note that $\langle \etab, \gammab\rangle = \rank \gammab$.
\end{proof}

\begin{corollary}\label{corollary:hom-vanishing for semistables}
Let $\alphaw \in \Knot (\Cc)$ be a positive algebraic invariant,
and let $\betab$ be a positive numerical invariant such that $\langle \alphaw, \betab\rangle = 0$. Let $\etaw \in \Knot(\Cc)$ be an effective algebraic invariant. Then there exists $m \gg  0$ such that for any $E \in \Mstackss(k)$ satisfying $\langle \etab, E_i\rangle \geq 0$  for every stable subquotient $E_i$ of $E$, a generic vector bundle $V$ with algebraic invariant $m \alphaw + \etaw$ satisfies
\[ \Hom(V,E) =  0.\]
\end{corollary}
\begin{proof}
The proof inductively considers the Jordan-H\"{o}lder filtration $0 \subsetneq E^{(1)} \subset \cdots \subsetneq E^{(r)} = E$ of $E$ whose subquotients $E_i = E^{(i)}/E^{(i-1)}$ are $\alphab$-stable. By applying \cref{proposition:Hom-vanishing} to each $E_i$ we deduce for $m \gg 0$ that a general vector bundle $V$ with algebraic invariant $m \alphaw + \etaw$ satisfies  $\Hom(V,E_i) = 0$ for each $ i = 1,\dots, r$.
By inductively applying $\Hom(V,-)$ to the exact sequences $0 \rightarrow E^{(i-1)} \rightarrow E^{(i)} \rightarrow E_i \rightarrow 0$, we obtain $\Hom(V,E) = 0$.
\end{proof}

\begin{remark}\label{rmk: eff bdd semiample}
In fact, the same effective bound for $m$ giving Hom-vanishing for stables given in \cref{prop:eff bdd semiample} also work for the Hom-vanishing of semistable vector bundles, as the proof of \cref{corollary:hom-vanishing for semistables} involves applying \cref{proposition:Hom-vanishing} to subinvariants.
\end{remark}

Thus semistability can be characterised in terms of a Hom-vanishing condition as follows.

\begin{proposition}\label{proposition:semistability equivalent to Hom-vanishing}
Let $\alphaw \in \Knot (\Cc)$ be a positive algebraic invariant,
and let $\betab$ be a positive numerical invariant such that $\langle \alphaw, \betab\rangle = 0$. 
Then $E \in \Mstack(k)$ is $\alphab$-semistable if and only if there is a vector bundle $V$ with algebraic invariant $m \alphaw$ for some $m > 0$ such that
\[ \Hom(V,E) = \Ext^1(V,E) = 0.\]
\end{proposition}
\begin{proof}
Note that the assumption $\langle \alphaw ,\betab \rangle=0$ gives $\dim \Hom(V,E) = \dim \Ext^1(V,E)$, so the latter statement is equivalent to simply requiring $\Hom(V,E) = 0$. The forward direction is \cref{corollary:hom-vanishing for semistables}. Conversely suppose $V$ has invariant $m\alphaw$ and satisfies $\Hom(V,E) =  0$. To show $E$ is $\alphab$-semistable we consider a subbundle $E' \subset  E$ with quotient $E''$ and apply $\Hom(V,-)$ to the exact sequence $0 \rightarrow E' \rightarrow E \rightarrow E'' \rightarrow 0$ to deduce $\Hom(V,E') = 0$. This implies $\langle m \alphab , [E'] \rangle  = - \dim \Ext^1(V,E') \leq 0 = \langle \alphab ,\betab \rangle$, from which we obtain $\mu_{\alphab}(E') \leq \mu_{\alphab}(E)$.
\end{proof}

\subsection{Ext-vanishing and separating stable bundles}
\label{sec:ext-vanishing}

Here we prove the key results that enable us to deduce ampleness of the determinantal line bundle. First, we use the fact that Serre duality sends semistable vector bundles to semistable vector bundles (see \cref{theorem:SD-stability}) to translate Hom-vanishing results into Ext-vanishing results. Throughout $\deltaw=(\deltab,\Oc_{\Cc}(x))$ denotes the numerical invariant of a degree $1$ torsion sheaf supported at a non-stacky point $x$, that is $\deltab=(0,1,\underline{0})$.

\begin{lemma}\label{lemma:Ext vanishing} Let $\alphaw \in \Knot (\Cc)$ be a positive algebraic invariant, and let $\betab$ be a positive numerical invariant such that $\langle \alphaw, \betab\rangle = 0$. Then for every $m > \kappa_{\betab}$ and for every $E \in \Mstackss(k)$, a general vector bundle $V$ with invariant $m\alphaw - \deltaw$ satisfies \[\Ext^1(V,E)=0.\]
\end{lemma}
\begin{proof} Note that $\Ext^1(V,E) = \Hom(V^\vee,\SD(E))^*$ and that $V^\vee$ has invariant $m\alphaw^\vee + \deltaw$. Note that: \begin{enumerate}[label=(\alph*)] \item  $\langle \deltab,  G \rangle = -\rank(G)  \leq 0$ for every bundle $G$.
    \item if $E$ is $\alphab$-semistable, then $\SD(E)$ is $\alpha^\vee$-semistable (see \cref{theorem:SD-stability});
    \item if $\langle \alphab, \betab \rangle =0$, then also $\langle \alphab^\vee, \SD(\betab) \rangle=0$ (see \eqref{eq: pairing and SD}).
\end{enumerate}
These three conditions ensure that we can apply \cref{corollary:hom-vanishing for semistables} (and \cref{rmk: eff bdd semiample}) to the $\alphab^\vee$-semistable sheaf $\SD(E)$ and conclude the argument.
\end{proof}

Before we can show that the determinantal line bundle has enough sections to separate most points, we first need a lemma which is a step towards producing the vector bundle defining the section we want: rather than constructing a vector bundle with algebraic invariant $m\alphaw$ we construct a vector bundle with invariant $m \alphaw - \deltaw$. We will later extend $V$ to construct the section needed to separate points.

\begin{lemma}\label{lemma:construction of V}
Let $E_0, \dots , E_\ell$ be $\alphab $-stable bundles whose numerical invariants $\betab_i$ satisfy $\langle \alphaw ,\betab_i \rangle =0$ and such that $E_0 \not\cong E_i$ for every $i=1,\dots, \ell$. Then for $m > \max_i \kappa^+_{2\betab_i}$, a generic  vector bundle $V$ with invariant $m \alphaw - \deltaw$ has the following properties:
\begin{enumerate}[label=(\roman*)]
\item \label{it:lem-i} $\Ext^1(V,E_i) = 0$ for all $i = 0,\dots, \ell$;
\item \label{it:lem-ii} for all $i = 0,\dots, \ell$ any non-zero homomorphism $V \rightarrow E_i$ is surjective;
\item \label{it:lem-iii} for all $i = 1,\dots , \ell$ and non-zero homomorphisms $f_0 \colon V \rightarrow E_0$ and $f_i \colon V \rightarrow E_i$, the homomorphism $f=(f_0,f_i) \colon V \rightarrow E_0 \oplus E_i$ is surjective.
\end{enumerate}
\end{lemma}
\begin{proof}
Part \ref{it:lem-i} follows from \cref{lemma:Ext vanishing}. For Part \ref{it:lem-ii}, we can apply \cref{proposition:Hom-vanishing} \ref{it:hom-van-2} to the $\alphab$-stable bundles $E_0, \dots, E_\ell$ to deduce that, for a $m > \kappa_{\betab}$ a generic vector bundle $V$ with invariant $m \alphaw - \deltaw$ the only non zero maps $V \to E_i$ are necessarily surjective.

For Part \ref{it:lem-iii}, we apply \cref{proposition:Hom-vanishing} \ref{it:hom-van-1} to $E_0 \oplus E_i$ to deduce that the image of $f$, which we denote $G_i$, has necessarily the same $\alphab$-slope of $E_0 \oplus E_i$. Note that, by \eqref{equation:slope and euler char} and the fact that $\langle \alphab, \beta_0 \rangle = 0 = \langle \alphab, \betab_i \rangle$, this slope coincide with $\mu_\alphab(E_0)=\mu_\alpha(E_i)$.
Since the map $f$ is non zero, this implies that if $f$ is not surjective, then $G_i$ is either isomorphic to either $E_0$ or $E_i$. Since both $f_0$ and $f_i$ are surjective, it follows that also the projections $G_i \to E_0$ and $G_i \to E_i$ are surjective. Since $E_0 \not\cong E_i$, the conditions $G_i \cong E_i$ or $G_i \cong E_0$ are impossible to achieve, thus the map $f$ is surjective.
\end{proof}

\begin{remark}\label{rmk effective bound for finite map to proj space}
Note the bound $m> \max_{i}\kappa_{2\betab_i}^+$, where this constant is given in \cref{prop:eff bdd semiample}. The factor $2$ arises, as we apply \cref{proposition:Hom-vanishing} to $E_0 \oplus E_i$, which has invariant $\betab_0 + \betab_i \leq \max_{i}2\betab_i$.
\end{remark}

To separate certain polystable vector bundles, we now construct a vector bundle $H$ with algebraic invariant $m \alphaw $ as a Hecke extension of a skyscraper sheaf at a non-stacky point $x \in \Cc$ by a vector bundle $V$ with algebraic invariant $m\alphaw - \deltaw$ as described before \cref{lemma:construction of V}.

\begin{proposition} \label{prop:separating polystables}
Let $E = E_1 \oplus \cdots \oplus E_\ell$ and $F \in F_1 \oplus \cdots \oplus F_{\ell'}$ be $\alphab $-polystable vector bundles on $\Cc$ with numerical invariant $\betab$. If $\langle \alphab ,\betab \rangle = 0$ and none of the stable summands $E_i$ are isomorphic to the stable summand $F_1$, then for $m > \kappa_{2\betab}^+$ there exists a vector bundle $H$ with algebraic invariant $m \alphaw $ such that
\begin{equation}\label{eq:separating polystables} \Hom(H,E) = 0 \quad \text{and} \quad \Hom(H,F) \neq  0.
\end{equation}
Hence there is a section of (a power of) the determinantal line bundle separating $E$ and $F$.
\end{proposition}
\begin{proof}
Let $\betab_i$ be the numerical invariants of $E_i$ and $r_i$ be the rank of $E_i$. Since up to a constant, the $\alphab $-slope of any vector bundle $G$ with invariant $\gammab$ is $\langle \alphab ,\gammab \rangle /\rk(\gammab)$ and we assumed $\langle \alphab ,\betab \rangle = 0$, we conclude $\langle \alphab ,\betab_i \rangle = 0$ for all $i$.

We start by applying \cref{lemma:construction of V} to the $\alphab $-stable vector bundles $E_0:= F_1$ and $E_1,\dots, E_\ell$ to deduce that a generic vector bundle $V$ with invariant $m\alphaw - \deltaw$ has the properties stated in this lemma. Indeed as $\deltaw$ is the invariant of a skyscraper sheaf supported at a non-stacky point, we have $\langle \deltab ,\betab_i \rangle = - \rank(\betab_i) = - r_i < 0$ which yields the necessary inequality to apply \cref{lemma:construction of V}.

Fix a non-zero surjection $\phi \colon V \twoheadrightarrow E_0$ and let $K$ denote the kernel, so that
\[ 0 \longrightarrow K \longrightarrow V \overset{\phi}{\longrightarrow} E_0 \longrightarrow 0
\] is an exact sequence. The idea of the proof is to enlarge $V$ to a vector bundle $H$ which is an extension of a skyscraper sheaf $\Oc_x$ at a non-stacky point by $V$ such that $\phi$ extends to a map $H \to E_0$, but for $i> 0$ no non-zero map $\psi_i \colon V \rightarrow E_i$ extends to a map $H \to E_i$. This would prove
\[ \Hom(H,E_0) \neq 0 \quad \mathrm{but} \quad \Hom(H,E_i) = 0 \:  \:\mathrm{for } \: i = 1,\dots, \ell \]
from which we deduce the claim \eqref{eq:separating polystables}. By Serre duality thus such an extension $H$ is determined by a line $L$ in the fibre $V_x$. We will find such a line $L$ by considering Grassmannians of quotients of the fibre $K_x$ of the kernel of $\phi$ at $x$.

More precisely, for $i = 1,\dots, \ell$, let $\mathrm{Gr}(K_x,r_i)$ be the Grassmannian of $r_i$-dimensional quotients of $K_x$, where $x$ is a fixed non-stacky point of $\Cc$. For each $i = 1,\dots, \ell$, we next construct a morphism
\[ q_i \colon  \mathbb{P}(\Hom(V,E_i)) \rightarrow \mathrm{Gr}(K_x,r_i). \]
For every non-zero morphism $\psi_i \colon V \rightarrow E_i$, we construct the commutative diagram
\[ \xymatrix{
0 \ar[r]  & K \ar@{..>}[d]^{\psi_i'} \ar[r]^\iota & V\ar[r]^\phi \ar[d]^{(\phi, \psi_i)} & E_0\ar[r] \ar[d]^= & 0\\
0 \ar[r] & E_i  \ar[r] & E_0 \oplus E_1\ar[r]  & E_0 \ar[r] & 0
}\]
where, by \cref{lemma:construction of V} the central map $(\phi, \psi_i)$ is surjective and so it induces the surjection $\psi'_i \colon K \to E_i$ by the snake lemma. In fact, we have an exact sequence
\[ 0 \longrightarrow K_i \longrightarrow K \longrightarrow E_i \longrightarrow 0, \]
where $K_i$ denotes the kernel of $(\phi,\psi_i) \colon V \rightarrow E_0 \oplus E_i$.  The map $q_i$ is thus defined by restricting $\psi'_i$ to $x$, that is
\[ q_i (\psi_i) \coloneqq  (\psi_i')_x \colon K_x \to (E_i)_x.
\]
The image of $q_i$ has dimension at most $\dim \Hom(V,E_i) - 1$. Since $\Ext^1(V,E_i) = 0$ by assumption and $\langle \alphab ,\betab_i \rangle =0$, we have
\[ \dim \Hom(V,E_i) = \langle m \alphab - \deltab,\betab_i \rangle = - \langle \deltab, \betab_i \rangle = r_i. \]
For a line $L$ in $K_x$, we define the Schubert variety $S_{L,i}$ by
\[S_{L,i} \coloneqq \{ f \colon K_x \twoheadrightarrow W \text{ such that } f(L) =0 \} \subset \mathrm{Gr}(K_x,r_i).\] 
The codimension of $S_{L,i} \subset \mathrm{Gr}(K_x,r_i)$ is $r_i$, whereas $\dim \mathrm{Im}(q_i) \leq r_i - 1$. Thus by the  Bertini--Kleiman Theorem, for a general line $L \subset K_x$, the image $\mathrm{Im}(q_i)$ is disjoint from the Schubert variety $S_{L,i}$ for $ i = 1,\dots, \ell$. Let us fix such a general line $L \subset K_x \subset V_x$. Observe that
\[ K_x = \Hom(K,\Oc_x)^\vee \cong \Ext^1(\Oc_x, K) \qquad \text{ and similarly} \qquad   V_x = \Hom(V,\Oc_x)^\vee \cong \Ext^1(\Oc_x, V)
\] where the isomorphisms are given by Serre duality. Thus, the line $L_x$ naturally defines two non-trivial extensions which fit in the diagram
\begin{equation} \label{eq:diagExt}  \xymatrix{
0 \ar[r]  & K \ar@{_{(}->}[d]^{\iota} \ar[r]  & G\ar[r] \ar@{_(->}[d]^{a} & \Oc_x \ar[r] \ar@{=}[d] & 0\\
0 \ar[r] & V  \ar[r]  & H \ar[r]  & \Oc_x \ar[r] & 0
}.\end{equation}
We note that $H$ is a vector bundle, as this extension is non-split and $H$ is an extension of the torsion sheaf $\mathcal{O}_x$ by $V$. We first of all see from \eqref{eq:diagExt} that the cokernels of the maps $\iota$ and $a$ are naturally isomorphic. Since by definition $K = \ker(\phi)$, we deduce that $\coker(\iota)=E_0$, and thus the map $\phi$ extends to $H$.

It remains to show that none of the maps $\psi_i$ extend to $H$.
Since $L$ avoids $S_{L,i}$, then the image of $L$ in $(E_i)_x$ via the map $\psi_i$ is non-zero, and thus, as done earlier, it defines a non-trivial extension $H_i$ that fits in the diagram:
\begin{equation*}\xymatrix{
0 \ar[r]  & V \ar@{->>}[d]^{\psi_i} \ar[r]  & H\ar[r] \ar@{->>}[d] & \Oc_x \ar[r] \ar@{=}[d] & 0\\
0 \ar[r] & E_i  \ar[r]  & H_i \ar[r]  & \Oc_x \ar[r] & 0
}.\end{equation*} If the map $\psi_i$ extended to $H$, then it would follow that the extension given by $H_i$ were split, which is a contradiction.
\end{proof}

\section{Ampleness of the determinantal line bundle and projectivity}
\label{sec:proj-of-gms}

We finally combine the previous two sections to prove \cref{ThmA} and \cref{ThmB}.

Throughout this section, we fix numerical invariants $\alphab$ and $\betab$ satisfying $\langle \alphab, \betab \rangle = 0$ and assume $\alphab$ is a generating numerical invariant. We recall that this vanishing of the Euler pairing can be arranged by applying \cref{lemma how to modify to orthogonal alpha}.  In order to uniquely (up to isomorphism) specify a determinantal line bundle we need to fix a line bundle $N$ on $C$ such that $\deg \pi_*N = \deg \pi_* \alphab$; then the determinantal line bundle $\mathcal{L}_V$ on $\Mstackss$ associated to a vector bundle $V$ with algebraic invariant $ \alphaw = (\alphab, N)$ is, up to isomorphism, independent of $V$ by \cref{proposition:determinantal line bundles dependence on V}. We denote the restriction of this determinantal line bundle to $\Mstackss$ by $\mathcal{L}_\alphaw$ (to emphasise that it depends on the algebraic class $\alphaw$ rather than the numerical class $\alphab$). Note that if $V$ is a vector bundle with algebraic invariant $m \alphaw$, then $\mathcal{L}_V \cong \mathcal{L}_\alphaw^{\otimes m}$.

\subsection{Global generation}

The following result gives a slight refinement of \cref{ThmA}.

\begin{theorem}\label{theorem:semiample}
Let $k$ be an arbitrary field, and assume $\langle \alphaw, \betab \rangle = 0$ with $\alphaw$ a generating algebraic invariant. Then the line bundle $\mathcal{L}_\alphaw$ on the stack $\Mstackss$ is semiample. More precisely, for every positive integer $m$ with
\[ m > (g_{\Cc} - 1)(\rank \betab)^2, \]
$\mathcal{L}_{\alphaw}^{\otimes m}$ is basepoint-free. If additionally $k$ has characteristic zero, then the line bundle $\mathcal{L}_{\alphaw}$ descends to a semiample line bundle $L_{\alphaw}$ on the good moduli space $\Mspacess$.
\end{theorem}
\begin{proof}
We can assume without loss of generality that $k$ is algebraically closed, as it suffices to know that the base change to an algebraic closure is semiample (see \cite[Exercise 19.2.I]{rising-sea}).

Fix a positive natural number $m$ such that $m > (g_{\Cc} - 1)(\rank \betab)^2$; this gives an effective bound for Hom-vanishing by \cref{prop:eff bdd semiample} and \cref{rmk: eff bdd semiample}. For a point of $\Mstackss$ corresponding to an $\alphab$-semistable vector bundle $F$ on $\Cc$ with numerical invariants $\betab$, we know that a general bundle $V$ with algebraic invariant $m \alphaw$  satisfies $\Hom(V,F) = 0$ by \cref{corollary:hom-vanishing for semistables}. In particular, we can find such a vector bundle $V$ so that the associated section $\sigma_V$ of $\mathcal{L}_V \cong \mathcal{L}_\alphaw^{\otimes m}$ is nonzero at this point by \eqref{prop:sigma}. Since $\Mstackss$ is quasi-compact, the non-vanishing loci of finitely many such sections cover $\Mstackss$ and so $\mathcal{L}_\alphaw^{\otimes m}$ is basepoint-free.

For the final claim, we let $\sigma_0,\dots, \sigma_n$ be global sections that generate $\mathcal{L}_\alphaw^{\otimes m}$ and thus induce a morphism $\phi_m \colon \Mstackss \rightarrow \PP^n$ such that $\mathcal{L}_\alphaw^{\otimes m} \cong \phi^* \Oc_{\PP^n}(1)$. Since the good moduli space map $f \colon \Mstackss \rightarrow \Mspacess$ is initial amongst morphisms to schemes, $\phi_m$ must factor via $f$ and so there is an induced morphism $\varphi_m \colon \Mspacess \rightarrow \PP^n$ such that $L_m:=\varphi_m^* \Oc_{\PP^n}(1)$ pulls back along $f$ to $\mathcal{L}_\alphaw^{\otimes m}$. Then ${L}_\alphaw:=L_{m+1} \otimes L_m^{-1}$ pulls back along $f$ to $\mathcal{L}_\alphaw$.
\end{proof}

\subsection{Ampleness and projectivity}

The following theorem together with \cref{cor:langton} proves \cref{ThmB}. 

\begin{theorem}\label{theorem: det is ample}
Let $k$ be a field of characteristic zero and assume $\langle \alphaw, \betab \rangle = 0$ where $\alphaw$ is a generating algebraic invariant. Then the line bundle $L_\alphaw$ on $\Mspacess$ is ample and $\Mspacess$ is projective.
\end{theorem}
\begin{proof}
Since $\Mspacess$ is proper, by the cohomological criterion for ampleness \cite[Tag \href{https://stacks.math.columbia.edu/tag/0D38}{0D38}]{stacks-project} and flat base change, we can assume without loss of generality that $k$ is algebraically closed.

As in \cref{theorem:semiample}, we know that a sufficiently large power $m$ of the determinantal line bundle on $\Mstackss$ is globally generated by finitely many sections which determine a morphism $\phi \colon  \Mstackss \rightarrow \PP^n$ that factors via the good moduli space map $f \colon \Mstackss  \rightarrow \Mspacess$ and a morphism $\varphi \colon \Mspacess \rightarrow \PP^n$ as in \cref{theorem:semiample}. Since $\Mspacess$ is proper, $\varphi$ is a proper morphism and to conclude the proof it is then enough to show that $\varphi$ is finite.

To show that $\varphi$ is finite, it suffices to show that the fibres of $\varphi$ are finite by \cite[Tag \href{https://stacks.math.columbia.edu/tag/0A4X}{0A4X}]{stacks-project}. Since $\Mspacess$ is of finite type, it is enough to check that fibres over $\bar{k}$-points are finite, so we will assume $k= \bar{k}$ (see \cite[Remark 12.16]{gortz-wedhorn}). We will show finiteness by supposing, for a contradiction, that there is a smooth proper connected curve $X$ and a non-constant map $\gamma \colon X \to \Mspacess$
  such that $\varphi \circ \gamma \colon X \to \PP^n$ is constant, so that any section of any power of $\gamma^* L_{\alphaw}$ is constant.  

The $k$-points of the good moduli space $\Mspacess$ correspond to the closed points of the stack $\Mstackss$, which are precisely the $\alphab$-polystable vector bundles on $\Cc$ with invariant $\betab$. For a given polystable bundle, there are only finitely many polystable bundles with the same invariants and isomorphic stable summands. Since the image of the non-constant curve $\gamma(X)(k)$ in $\Mspacess(k)$ is infinite, there must be two points corresponding to polystable bundles $E$ and $F$ such that one of the stable summands of $F$ does not appear in $E$. Then by \cref{prop:separating polystables} for every $m > \kappa_{2\betab}^+$ there exists a vector bundle $H$ with algebraic invariants $m \alphaw$ such that
\[ \Hom(H,E) = 0 \quad \text{and} \quad \Hom(H,F) \neq  0.\]
The vector bundle $H$ determines a section $\sigma_H$ of $\mathcal{L}_{\alphaw}^{\otimes m}$ that separates these points: $\sigma_H(E) \neq 0$ and $\sigma_H(F) = 0$.
Since $\mathcal{L}_{\alphaw}^{\otimes m}  = \phi^* \Oc_{\PP^n}(1) = f^* \varphi^* \Oc_{\PP^n}(1)$, we can write $\sigma_H = f^*\sigma$ for a section $\sigma$ of $L_{\alphaw}^{\otimes m}$. Then $\gamma^*\sigma$ is a non-constant section of $\gamma^* L_{\alphaw}^{\otimes m}$ giving the desired contradiction.

Since $\varphi \colon \Mspacess \rightarrow \PP^n$ is a finite morphism of proper schemes, we can conclude that the ample line bundle $\Oc_{\PP^n}(1)$ pulls back to an ample line bundle, which is a power of $L_\alphaw$, and thus $L_\alphaw$ is ample and $\Mspacess$ is projective.
\end{proof}

\begin{corollary}\label{cor:eff bdd for finite map to Pn} For every positive integer $m$ satisfying the inequality
\[m > \kappa_{2\betab}^+ = 4 \left(g_{\Cc} - 1 + \frac{1}{2\rank\betab}\right)(\rank \betab)^2,\] the line bundle $L_{\alphaw}^{\otimes m}$ induces a finite morphism from $\Mspacess$ to a projective space.
\end{corollary}
\begin{proof}
This follows from the proof of \cref{prop:separating polystables} combined with Remark \ref{rmk effective bound for finite map to proj space}.
\end{proof}

\appendix

\section{Grothendieck duality for determinantal line bundles}

In this short appendix, we give a strengthening of Grothendieck duality for stacks (\cref{theorem: Grothendieck duality} below strengthens \cite[Corollary 4.15]{hall-rydh:2017:perfect-complexes}) and then we use this to prove \cref{proposition:determinantal line bundles dependence on V}.

\subsection{Grothendieck duality for algebraic stacks}
\label{appendix: GD}

For the lack of unified exposition, we first summarise known properties of derived categories of algebraic stacks in order to prove a strengthening of Grothendieck duality for stacks.

Given a morphism of algebraic stacks $f \colon \Xc \to \Yc$, there are two ways to construct derived pullback. Let $\Dqc(\Xc)$ be the unbounded derived category of $\Oc_\Xc$-modules with quasi-coherent cohomology.  
In what follows, we will write $\Lb f^*$ to denote the functor $\Lb f^*_{\mathrm{qc}}$ in \cite[\S1.3]{hall-rydh:2017:perfect-complexes} (see also \cite{olsson:2007} and \cite{laszlo-olsson:2008}).
This functor automatically admits a right adjoint $\Rb f_*$ (denoted $\Rb (f_{\mathrm{qc}})_*$ in \cite[\S1.3]{hall-rydh:2017:perfect-complexes}).
The other approach to pullback may give as its adjoint a different derived pushforward functor, which nevertheless coincides with $\Rb f_*$ when restricted to the bounded below derived category \cite[Lemma 1.2(2)]{hall-rydh:2017:perfect-complexes}.
For any $K, M \in \Dqc(\Xc)$, there is the derived tensor product $K \otimes^\Lb P$ over $\Oc_\Xc$, as well as its right adjoint $\localRHom(K,M) \in \Dqc (\Xc)$, which in \cite[\S1.2]{hall-rydh:2017:perfect-complexes} is denoted by $\localRHom^{\mathrm{qc}}(K,M)$.
For any $K \in \Dqc (\Xc)$, we denote by $K^\vee = \localRHom(K, \Oc_\Xc)$ its derived dual. 

\begin{lemma}
\label{lemma: basic dualities}
	Let $\Xc$ be an algebraic stack. Fix $K, M, N \in \Dqc (\Xc)$ and a perfect complex  $P \in \Dqc (\Xc)$ (e.g.\ $P$ is a vector bundle). Then we have the following natural isomorphisms.
	\begin{enumerate}[label=\emph{(\roman*)}]
		\item \label{item: tensor-Hom adjunction}
		$\localRHom(K \otimes^\Lb M, N) \cong
		\localRHom(K, \localRHom(M, N))$.
		\item \label{item: switching sides for perfect complexes}
		$\localRHom(K, P \otimes^\Lb N)
		\cong \localRHom (K \otimes^\Lb (P^\vee) , N)$.
	\end{enumerate}
\end{lemma}

\begin{proof}
	Part \emph{\ref{item: tensor-Hom adjunction}} follows from the Yoneda Lemma, \cite[Eq. 1.4]{hall-rydh:2017:perfect-complexes} and the last equation in \cite[\S1.2]{hall-rydh:2017:perfect-complexes}.
	Part \emph{\ref{item: switching sides for perfect complexes}} follows from Part \emph{\ref{item: tensor-Hom adjunction}} and \cite[Lemma 4.3(2)]{hall-rydh:2017:perfect-complexes}.
\end{proof}

The theory explained in \cite{hall-rydh:2017:perfect-complexes} can be summarised in a slogan: derived pushforwards work best with concentrated morphisms. Although we will not define a \textbf{concentrated morphism} (the reader is referred to \cite[Definition 2.4]{hall-rydh:2017:perfect-complexes}), we will explain that the relevant morphisms that we consider are of this kind, and note that they are by definition quasi-compact.
We say that a stack $\Xc$ is \textbf{concentrated} if $\Xc \to \Spec \ZZ$ is concentrated. Further, we will contract the properties quasi-compact and quasi-separated to simply \textbf{qcqs}.

\begin{theorem}[{\cite[Theorem 1.4.2]{drinfeld-gaitsgory:2013}, \cite[Theorem C]{hall-rydh:2015}}]
\label{theorem: when qcqs stacks are concentrated}
	Let $\Xc$ be a qcqs algebraic stack over a field of characteristic $0$. If $\Xc$ has affine stabilisers, then $\Xc$ is concentrated.
\end{theorem}

By {\cite[Lemma 2.5(1)]{hall-rydh:2017:perfect-complexes}}, concentrated morphisms are stable under pullback.

\begin{theorem}[{Projection formula, \cite[Corollary 4.12]{hall-rydh:2017:perfect-complexes}}]
\label{theorem: projection formula}
	Let $f \colon \Xc \to \Yc$ be a concentrated morphism of algebraic stacks. Then the following natural map is an isomorphism for all $K \in \Dqc (\Xc)$, $M \in \Dqc (\Yc)$:
	\[
	(\Rb f_* K) \otimes^\Lb M \to \Rb f_* ( K \otimes^\Lb \Lb f^* M )
	. \]
\end{theorem}

\begin{theorem}[{Tor-independent base change, \cite[Corollary 4.13]{hall-rydh:2017:perfect-complexes}}]
\label{theorem: base change}
	Consider the following Cartesian square of algebraic stacks, where $x$ is concentrated.
	\[\begin{tikzcd}
		{\Xc'} & \Xc \\
		{\Yc'} & \Yc
		\arrow["{q'}", from=1-1, to=1-2]
		\arrow["x", from=1-2, to=2-2]
		\arrow["q"', from=2-1, to=2-2]
		\arrow["{x'}"', from=1-1, to=2-1]
	\end{tikzcd}\]
	If $x$ and $q$ are tor-independent (e.g.\ if one of them is flat), then the base change morphism is an isomorphism for any $K \in \Dqc (\Xc)$:
	\[
	\Lb q^* \Rb x_* (K) \to \Rb x'_* \Lb q'^* (K).
	\]
\end{theorem}

\begin{lemma}
\label{lemma: Dqc is rcg}
	Let $\Xc$ be a qcqs concentrated algebraic stack of characteristic $0$ that admits a good moduli space.
	Then $\Dqc (\Xc)$ is rigidly compactly generated, and the properties of being perfect, compact and dualisable are equivalent in $\Dqc (\Xc)$.
\end{lemma}

 \begin{proof}
  The notion of rigidly compactly generated category can be found in e.g. \cite[Reminder 2.2]{neeman:2021}. In order to satisfy this condition, we need to check that $\Dqc (\Xc)$ is compactly generated, tensor product commutes with coproducts, and that the compact objects in $\Dqc (\Xc)$ are precisely the dualisable complexes.
 	Since $\Xc$ admits a good moduli space, it is of s-global type (see \cite[p. 1]{hall-rydh:2017:perfect-complexes} for the definition), hence
 	by \cite[Theorem B]{hall-rydh:2017:perfect-complexes}, the category $\Dqc (\Xc)$ is compactly generated.
 	By definition of tensor product in algebraic geometry, it commutes with coproducts.
 	By \cite[Lemmas 4.3 and 4.4]{hall-rydh:2017:perfect-complexes}, the properties of being perfect, compact and dualisable are equivalent in $\Dqc (\Xc)$ under our assumptions.
 \end{proof}

\begin{lemma}
\label{lemma: existence of right adjoint of pushforward}
	Let $f \colon \Xc \to \Yc$ be a concentrated morphism of algebraic stacks.
	\begin{enumerate}[label=\emph{(\roman*)}]
		\item \label{item: right adjoint f^x of pushforward}
		$\Rb f_* \colon \Dqc (\Xc) \to \Dqc (\Yc)$ admits a right adjoint $f^\times$.
		\item \label{item: formula for f^x of a perfect complex}
		For a perfect $P \in \Dqc (\Yc)$, we have a natural isomorphism $f^\times (\Oc_\Yc) \otimes^\Lb \Lb f^* (P)  \cong f^\times (P)$.
	\end{enumerate}
\end{lemma}

\begin{proof}
	Part \emph{\ref{item: right adjoint f^x of pushforward}} is \cite[Theorem 4.14(1)]{hall-rydh:2017:perfect-complexes}. For Part \emph{\ref{item: formula for f^x of a perfect complex}}, we apply the Yoneda Lemma to the following sequence of isomorphisms:
	\begin{align*}
	 \RHom_\Xc ( \blank, f^\times (\Oc_\Yc) \otimes^\Lb \Lb f^* (P))
		& \cong \RHom_\Xc ( \blank \otimes^\Lb \Lb f^* (P)^\vee, f^\times (\Oc_\Yc) )
		, \qquad \text{\small \cref{lemma: basic dualities} \emph{\ref{item: switching sides for perfect complexes}}}
		\\ & \cong
		\RHom_\Yc \left(
			\Rb f_* (\blank \otimes^\Lb \Lb f^* (P)^\vee) , \Oc_\Yc
			\right)
		, \qquad \text{\small \cref{lemma: existence of right adjoint of pushforward} \emph{\ref{item: right adjoint f^x of pushforward}}}
		\\ & \cong
		\RHom_\Yc \left(
			(\Rb f_*\blank) \otimes^\Lb P^\vee ,  \Oc_\Yc
			\right)
		, \qquad \text{\small \cref{theorem: projection formula}}
		\\ & \cong
		\RHom_\Yc \left(
		\Rb f_*(\blank) , P
		\right)
		, \qquad \text{\small \cref{lemma: basic dualities} \emph{\ref{item: switching sides for perfect complexes}}}
		\\ & \cong
		\RHom_\Xc \left( \blank , f^\times (P)
		\right)
		, \qquad \text{\small \cref{lemma: existence of right adjoint of pushforward} \emph{\ref{item: right adjoint f^x of pushforward}}.}
		\qedhere
	\end{align*}
\end{proof} 

	We remark that we define $f^\times$ as the right adjoint to pushforward, and that this functor in general does not agree with the exceptional pullback $f^!$ that is part of Grothendieck's 6-functor formalism. However if we assume that $f$ is proper, then $f^\times \cong f^!$.

\begin{lemma}
\label{lemma: THE dualising complex commutes with pullback}
	Consider the following tor-independent Cartesian square of qcqs concentrated algebraic stacks that all admit good moduli spaces.
	\[\begin{tikzcd}
		{\Xc'} & \Xc \\
		{\Yc'} & \Yc
		\arrow["{q}", from=1-1, to=1-2]
		\arrow["\bar{p}", from=1-2, to=2-2]
		\arrow["\bar{q}"', from=2-1, to=2-2]
		\arrow["{p}"', from=1-1, to=2-1]
	\end{tikzcd}\]
	If $\bar p_*$ is concentrated and
	$\Rb \bar p_*$ sends compact objects to compact objects,
	then the natural transformation $\Lb q^* \bar p^\times \to p^\times \Lb \bar q$ is an isomorphism.
\end{lemma}

\begin{proof}
	By {\cite[Lemma 2.5(1)]{hall-rydh:2017:perfect-complexes}} and \cref{lemma: existence of right adjoint of pushforward}, we deduce that $\bar p^\times$ and $p^\times$ exist.

	Since the square by assumption is tor-independent, $\Xc'$ is isomorphic to the derived fibre product, and hence the
	argument of \cite[Proposition 3.24]{ben-zvi-francis-nadler:2010} applies and proves that
	the category $\Dqc (\Xc')$ is generated by compact objects of the form $\Lb q^* K \otimes^\Lb \Lb p^* M$, where $K \in \Dqc (\Xc)$ and $M \in \Dqc (\Yc')$ are compact.
	Hence the category of compact objects in $\Dqc (\Xc')$ is classically generated by objects of this form, by \cite[Theorem 2.1.2]{bondal-vdb:2003},
	and therefore, to prove that $\Rb p_*$ sends compact objects to compact objects, we just need to prove that $\Rb p_* (\Lb q^* K \otimes^\Lb \Lb p^* M)$ is compact for every compact $K \in \Dqc (\Xc)$ and $M \in \Dqc (\Yc')$.
	By the projection formula (\cref{theorem: projection formula}) and base change (\cref{theorem: base change}), we have
	$\Rb p_* (\Lb q^* K \otimes^\Lb \Lb p^* M) =
	(\Rb p_* \Lb q^* K) \otimes^\Lb M =
	(\Lb \bar q^* \Rb \bar p_* K) \otimes^\Lb M$.
	But $\Rb \bar p_* K$ is compact by assumption, hence perfect by \cref{lemma: Dqc is rcg}, which implies that $\Lb \bar q^* \Rb \bar p_* K$ is perfect and compact, and the same holds for $(\Lb \bar q^* \Rb \bar p_* K) \otimes^\Lb M$.

	With this, and also by \cref{lemma: Dqc is rcg} and \cite[Proposition 5.3, Definition 5.4]{neeman:2021},
	we can apply \cite[Proposition 6.3]{neeman:2021} to the induced square of derived categories.
\end{proof}

Let us give a strengthening of \cite[Corollary 4.15]{hall-rydh:2017:perfect-complexes}, which is formulated in a much more restrictive setting: in particular, their morphism is required to be finite.

\begin{theorem}[Grothedieck duality for algebraic stacks]
	\label{theorem: Grothendieck duality}
	Let $p \colon \Xc \to \Yc$ be a concentrated morphism of qcqs concentrated algebraic stacks that admit good moduli spaces.
	If $\Rb p_*$ sends compact objects to compact objects,
	then for any $K \in \Dqc (\Xc)$ and a perfect $P \in \Dqc (\Yc)$,
	there is a natural isomorphism
	\[
	\Rb p_* \localRHom_\Xc (K , P \otimes^\Lb p^\times \Oc_\Yc) \cong \localRHom_\Yc (\Rb p_* K, P),
	\]
	and the formation of $p^\times \Oc_\Yc$ commutes with tor-independent base change to a qcqs concentrated algebraic stack that admits a good moduli space.
\end{theorem}

\begin{proof}
	For any test object $T \in \Dqc (\Yc)$, we will apply the functor $\Hom_\Yc(T, \blank)$ to the desired isomorphism and observe that we indeed get an isomorphism using \cref{lemma: basic dualities}\emph{\ref{item: tensor-Hom adjunction}} and \cref{theorem: projection formula}. By the Yoneda Lemma, we conclude the desired formula for Grothendieck duality.
	By \cref{lemma: THE dualising complex commutes with pullback}, the formation of $f^\times \Oc_\Yc$ commutes with pullback as in the statement.
\end{proof}

\subsection{Proof of the determinantal line bundle formula}
\label{appendix: proof of det line bundle formula}

The aim of this section is to prove \cref{proposition:determinantal line bundles dependence on V}.

\begin{lemma}
	\label{lemma: determinantal line bundle from a skyscraper}
	For a closed point $x\in |\Cc| = |C|$, let $\Oc_x = \pi^*\Oc_{C, x}$ be the skyscraper sheaf on $\Cc$ supported on $x$. Fix a generating $\alphab \in \Knum (\Cc)$.
	Then $\Lc_{\Oc_x} = \det \Uc_{\betab, x}$ over $\Mstackss$, where
	$\Uc_{\betab, x} = \Uc_{\betab} |_{\{x\} \times \Mstackss}$ is the restriction of the universal bundle.
\end{lemma}

\begin{proof}
	Let $x \colon \Spec k \to \Cc$ denote the point $x$ of $\Cc$ and consider the following diagram.
	\[\begin{tikzcd}
		{\Spec k} & \Mstackss \\
		\Cc & {\Cc \times \Mstackss} \\
		{\Spec k} && \Mstackss
		\arrow["{q'}"', from=1-2, to=1-1]
		\arrow["x"', from=1-1, to=2-1]
		\arrow["{x'}", from=1-2, to=2-2]
		\arrow["q"', from=2-2, to=2-1]
		\arrow["{=}", curve={height=-18pt}, from=1-2, to=3-3]
		\arrow["p", from=2-2, to=3-3]
		\arrow["{\bar p}", from=2-1, to=3-1]
		\arrow["{\bar q}"', from=3-3, to=3-1]
	\end{tikzcd}\]
	The proof is a direct calculation, for which we make a few observations. First, by  \cref{proposition: Mstackss of stacky curve is finite type}, the stack $\Mstackss$ is finite type with affine diagonal and affine stabilisers, and thus by \cref{theorem: when qcqs stacks are concentrated}, it is concentrated. Further, as concentrated morphisms are stable under pullback, we conclude that $p$ is a concentrated morphism.
	Notice that several functors do not need deriving: $x_* = \Rb x_*$, $q^* = \Lb q^*$, $p^* = \Lb p^*$, and that $V \otimes (\blank) = V \otimes^\Lb (\blank)$ for a vector bundle $V$.
	By Serre duality (\cref{theorem: Serre duality}), we have $\bar p^\times \Oc_{\Spec k} \cong \omega_\Cc[1]$ and thus  $p^\times (\Oc_{\Mstackss}) \cong q^* \omega_{\Cc} [1]$ by \cref{lemma: THE dualising complex commutes with pullback}.
	\begin{align*}
		\Lc_{\Oc_x}
		&=
		\det (\Rb p_* \localRHom (q^* x_* \Oc_x , \Uc_\betab))^\vee
		\\ &=
		\det (\Rb p_* \localRHom (x'_*  \Oc_{\Mstackss}, \Uc_\betab))^\vee
		, \qquad \text{by \cref{theorem: base change}}
		\\ &=
		\det \left( \Rb p_*
		\localRHom
		\left(
		\Uc_\betab^\vee \otimes p^\times \Oc_{\Mstackss} \otimes x'_*  \Oc_{\Mstackss},
		p^\times \Oc_{\Mstackss} \right)
		\right)^\vee
		, \qquad \text{by \cref{lemma: basic dualities}\emph{\ref{item: switching sides for perfect complexes}}}
		\\ &=
		\det \left( \Rb p_*
		\localRHom
		\left(
		x'_* (x'^* \Uc_\betab^\vee \otimes x'^* p^\times \Oc_{\Mstackss}),
		p^\times \Oc_{\Mstackss} \right)
		\right)^\vee
		, \qquad \text{by \cref{theorem: projection formula}}
		\\ &=
		\det \left(
		\localRHom_{\Mstackss}
		\left(
		\Rb p_*
		x'_* \left(x'^* \Uc_\betab^\vee \otimes  x'^* q^* \omega_\Cc[1] \right)
		,
		\Oc_{\Mstackss}
		\right)
		\right)^\vee
		, \qquad \text{by \cref{theorem: Grothendieck duality}}
		\\ &=
		\det \left(
		\localRHom_{\Mstackss}
		\left(
		x'^* \Uc_\betab^\vee [1]
		,
		\Oc_{\Mstackss}
		\right)
		\right)^\vee
		, \qquad \text{as } px'=\mathrm{id} \text{ and } x'^* q^* \omega_\Cc = q'^* x^* \omega_\Cc = \Oc_{\Mstackss}
		\\ &=
		\det (x'^* \Uc_\betab^\vee [1])
	    =
		\det (x'^* \Uc_\betab).
		\qedhere
	\end{align*}
\end{proof}

Using this computation, we are finally able to prove \cref{proposition:determinantal line bundles dependence on V}.

\begin{proof}[Proof of \cref{proposition:determinantal line bundles dependence on V}]
	Since the assignment $V \mapsto \Lc_V$ defines a group homomorphism $\mathrm K_0(\Cc) \to \pic{\Mstackss}$, we can write $\Lc_V = \Lc_{V'} \otimes \Lc_{[V]-[V']}$.
	By assumption, $[V] = [V']$ in $\Knum (\Cc)$, so $[V] - [V'] = [\det \pi_* V] - [\det \pi_* V']$, which can be rewritten as an integer combination of points of $C$ or equivalent of $\Cc$ via the identification $\pi \colon |\Cc| \to |C|$.
	The result now follows from \cref{lemma: determinantal line bundle from a skyscraper}.
\end{proof}

\bibliographystyle{alpha}
\bibliography{biblio}

\vspace{1cm}
\setlength\parindent{0pt}
\setlength\parskip{4pt}

\href{mailto:chiara.damiolini@utexas.edu}{\texttt{chiara.damiolini@austin.utexas.edu}} \\
University of Texas at Austin, Austin, USA

\href{mailto:v.hoskins@math.ru.nl}{\texttt{v.hoskins@math.ru.nl}} \\
Radboud University, Nijmegen, Netherlands

\href{mailto:murmuno@gmail.com}{\texttt{murmuno@gmail.com}} \\
Universität Duisburg--Essen, Essen, Germany

\href{mailto:l.taams@math.ru.nl}{\texttt{l.taams@math.ru.nl}} \\
Radboud University, Nijmegen, Netherlands

\end{document}